\documentclass[a4paper,11pt,twoside]{article}
\usepackage{amsmath,amssymb,amscd}
\usepackage{ascmac}
\usepackage[all]{xy}
\usepackage{enumerate}
\usepackage{cancel}
\usepackage[dvips]{graphicx}
\usepackage{authblk}

\usepackage{theorem}
\newtheorem{thm}{Theorem}[section]
\newtheorem{pro}[thm]{Proposition}
\newtheorem{`thm'}[thm]{``Theorem''}

\newtheorem{lem}[thm]{Lemma}

{\theorembodyfont{\normalfont}
\newtheorem{rmk}[thm]{Remark}
\newtheorem{rmks}[thm]{Remarks}
\newtheorem{dfn}[thm]{Definition}

\newtheorem{fact}[thm]{Fact}
\newtheorem{asm}[thm]{Assumption}
}
\makeatletter
\newenvironment{proof}[1][\proofname]{\par
  \normalfont
  \topsep6\p@\@plus6\p@ \trivlist
  \item[\hskip\labelsep{\textit{\mdseries #1}\@addpunct{\mdseries.}}]\ignorespaces
}{%
  \QED \endtrivlist
}
\newcommand{\proofname}{\normalfont{\textit{Proof.}}}
\makeatother
\makeatletter
\def\BOXSYMBOL{\RIfM@\bgroup\else$\bgroup\aftergroup$\fi
  \vcenter{\hrule\hbox{\vrule height.85em\kern.6em\vrule}\hrule}\egroup}
\makeatother
\newcommand{\BOX}{%
  \ifmmode\else\leavevmode\unskip\penalty9999\hbox{}\nobreak\hfill\fi
  \quad\hbox{\BOXSYMBOL}}
\newcommand\QED{\BOX}

\numberwithin{equation}{section}

{\theorembodyfont{\normalfont}

}

\newcommand{\bb}[1]{\mathbb#1}
\newcommand{\fra}[1]{\mathfrak#1}
\newcommand{\ca}[1]{\mathcal#1}
\newcommand{\innpro}[3]{\langle#1,#2\rangle_{#3}}

\newcommand{\alg}{{\rm alg}}
\newcommand{\ind}{{\rm ind}}
\newcommand{\id}{{\rm id}}

\newcommand{\End}{{\rm End}}
\newcommand{\dom}{{\rm dom}}
\newcommand{\op}{{\rm op}}
\newcommand{\Lie}{{\rm Lie}}
\newcommand{\fin}{{\rm fin}}

\newcommand{\Op}{{\rm Op}}
\newcommand{\bra}[1]{\left(#1\right)}
\newcommand{\bbra}[1]{\left\{#1\right\}}
\newcommand{\bbbra}[1]{\left[#1\right]}
\newcommand{\Dirac}{\cancel{\partial}}

\begin{document}
\title{$LT$-equivariant Index from the Viewpoint of $KK$-theory}
\author{Doman Takata}
\affil{{\small Department of mathematics, Kyoto university \\
  Kitashirakara, Oiwake-cho, Sakyo-ku \\
  Kyoto-shi, 606-8502, Japan
 \\E-mail address: {\tt d.takata@math.kyoto-u.ac.jp}}}
\date{\today}

\maketitle
\begin{abstract}
Let $T$ be a circle group, and $LT$ be its loop group. We hope to establish an index theory for infinite-dimensional manifolds which $LT$ acts on, including Hamiltonian $LT$-spaces, from the viewpoint of $KK$-theory. We have already constructed several objects in the previous paper \cite{T}, including a Hilbert space $\ca{H}$ consisting of ``$L^2$-sections of a Spinor bundle on the infinite-dimensional manifold'', an ``$LT$-equivariant Dirac operator $\ca{D}$'' acting on $\ca{H}$, a ``twisted crossed product of the function algebra by $LT$'', and the ``twisted group $C^*$-algebra of $LT$'', without the measure on the manifolds, the measure on $LT$ or the function algebra itself. However, we need more sophisticated constructions. In this paper, we study the index problem in terms of $KK$-theory.

Concretely, we focus on the infinite-dimensional version of the latter half of the assembly map defined by Kasparov. Generally speaking, for a $\Gamma$-equivariant $K$-homology class $x$, the assembly map is defined by $\mu^\Gamma(x):=[c]\otimes j^\Gamma(x)$, where $j^\Gamma$ is a $KK$-theoretical homomorphism, $[c]$ is a $K$-theory class coming from a cut-off function, and $\otimes$ denotes the Kasparov product with respect to $\Gamma\ltimes C_0(X)$. We will define neither the $LT$-equivariant $K$-homology nor the cut-off function, but we will indeed define the $KK$-cycles $j^{LT}_\tau(x)$ and $[c]$ directly, for a virtual $K$-homology class $x=(\ca{H},\ca{D})$ which is mentioned above. As a result, we will get the $KK$-theoretical index $\mu^{LT}_\tau(x)\in KK(\bb{C},LT\ltimes_\tau \bb{C})$. We will also compare $\mu^{LT}_\tau(x)$ with the analytic index $\ind_{LT\ltimes_\tau\bb{C}}(x)$ which will be introduced.


\end{abstract}
\tableofcontents
\section{Introduction}
The overall goal of our research is a construction of an index theorem for some infinite-dimensional manifolds, in terms of noncommutative differential geometry (\cite{ASi1}, \cite{ASi2}, \cite{Con94} and \cite{Kas88}). We will study a certain generalization of Hamiltonian $LT$-spaces (\cite{AMM}, \cite{Mei} and \cite{Son}).

\begin{asm}\label{setting}
Let $T$ be a circle group, and $LT:=C^\infty(S^1,T)$ be its loop group. It acts on the dual Loop algebra $L\fra{t}^*$, via the gauge action $l.A:=A+l^{-1}dl$. Let $\ca{M}$ be an infinite-dimensional manifold equipped with a smooth $LT$-action, a proper equivariant map $\Phi:\ca{M}\to L\fra{t}^*$, and an $LT$-equivariant Spionor bundle $\ca{S}\to \ca{M}$. Moreover, we suppose that $\ca{M}$ admits a $\tau$-twisted $LT$-equivariant line bundle $\ca{L}\to\ca{M}$, where ``$\tau$-twisted $LT$-equivariant'' means that a $U(1)$-central extension $LT^\tau$ (of $LT$) acts on $\ca{L}$, and the added center $U(1)$ acts on $\ca{L}$ by the scalar multiplication. The explicit definition of $\tau$ will be given in Definition \ref{central extension}.
\end{asm}

\begin{rmk}
We hope to deal with ``good'' representations satisfying several natural conditions, and such representations are called positive energy representations, extensively studied in \cite{PS} and \cite{FHTII}. But there is no such representations unless we consider a $U(1)$-central extension.


\end{rmk}

We have proved the following theorem in the previous paper \cite{T}.
Let us review this result in the language of this paper.

\begin{thm}[\cite{T}]

In the above situation, we can construct a $C^*$-algebra $A$ which can be regarded as ``$LT\ltimes_{-\tau} C_0(\ca{M})$'', a Hilbert space $\ca{H}$ which can be regarded as ``$L^2(\ca{M},\ca{L}\otimes\ca{S})$'', and an unbounded operator $\ca{D}:\dom(\ca{D})\to \ca{H}$ which can be regarded as an ``$LT^\tau$-equivariant Dirac operator on $\ca{M}$''. Moreover, $A$ acts on $\ca{H}$, and they define a spectral triple.

In addition, we can construct a $C^*$-algebra $LT\ltimes_{-\tau} \bb{C}$ which can be regarded as the ``$-\tau$-twisted group $C^*$-algebra of $LT$''.
$\ca{D}$ has a well-defined index valued in $R^\tau(LT)\cong KK(\bb{C},LT\ltimes_{-\tau} \bb{C})$. 

\end{thm}

It is natural to ask if the index has a ``cohomological formula'' just like the classical index theorem (\cite{ASe1}, \cite{ASi2}). According to \cite{Con94}, \cite{Kha} or \cite{Pus}, it is enough to describe the analytic index map in terms of Kasparov product in $KK$-theory, as a map between $KK$-groups. The followings must be useful not only the classical cases, but also for ours.

\begin{dfn}[\cite{Kas88}]
$(1)$ For a group $\Gamma$ satisfying the assumption in Remark \ref{ass. of grp}, and $\Gamma$-$C^*$-algebras $A$ and $B$, the $j$-homomorphism $j^\Gamma:KK_\Gamma(A,B)\to KK(\Gamma\ltimes A,\Gamma\ltimes B)$ is defined by using the reduced crossed products.

$(2)$ Let $X$ be a locally compact Hausdorff space equipped with a proper and cocompact $\Gamma$-action. 
An action $\Gamma\curvearrowright X$ is said to be proper if the map $X\times\Gamma\ni (x,\gamma)\mapsto (x,\gamma.x)\in X\times X$ is proper, and cocompact if the quotient space $X/\Gamma$ is compact.
If $\Gamma$ acts on $X$ properly and cocompactly, there exists a nonnegative, compactly supported function $c:X\to \bb{R}_{\geq0}$ satisfying that $\int_\Gamma c(\gamma.x)d\gamma=1$ for any $x\in X$. It induces a $K$-theory element $[c]\in KK(\bb{C},\Gamma\ltimes C_0(X))$ defined by $[c](\gamma)=\sqrt{c}\cdot\sqrt{\gamma.c}$. We call $[c]$ the Mishchenko line bundle associated to $\Gamma\curvearrowright X$.

$(3)$ The assembly map $\mu^\Gamma$ is defined by the composition of 
$$KK_\Gamma(C_0(X),\bb{C})\xrightarrow{j^\Gamma} KK(\Gamma\ltimes C_0(X),\Gamma\ltimes\bb{C})\xrightarrow{[c]\otimes-}KK(\bb{C},\Gamma\ltimes\bb{C}).$$
\end{dfn}

Let $\Gamma$ be a group as in Remark \ref{ass. of grp}, $X$ be a complete $Spin^c$-manifold, equipped with a Spinor bundle $W$ and a Dirac operator $D$. Let us consider the analytic index of the equivariant $K$-homology class $(L^2(X,W),D)\in KK_\Gamma(C_0(X),\bb{C})$.
\begin{dfn}[\cite{Kas88}]\label{def of a-ind}
For $s_1,s_2\in C_c(X,W)$, we can define a $C_c(\Gamma)$-valued inner product by
$$\innpro{s_1}{s_2}{\Gamma\ltimes\bb{C}}(\gamma):=\int \innpro{s_1(x)}{\gamma.(s_2(\gamma^{-1}.x)}{W}dx,$$
and a right action of $C_c(\Gamma)$ on $L^2(X,W)$ by
$$s*b(x):=\int_\Gamma \gamma.s(\gamma^{-1}.x)b(\gamma^{-1})d\gamma.$$
The analytic index $\ind_{\Gamma\ltimes \bb{C}}(L^2(X,W),D)\in KK(\bb{C},\Gamma\ltimes \bb{C})$ is given by the completion of $C_c(X,W)$ with respect to the above inner product.
\end{dfn}

\begin{fact}[\cite{Kas15}]
In the above situation,
$\ind_{\Gamma\ltimes\bb{C}}([L^2(X,W),D])=\mu^\Gamma([L^2(X,W),D]).$

\end{fact}

Unfortunately, our previous work does not corresponds to Kasparov's. More precisely, according to \cite{Kas15}, a $K$-homology class of a crossed product algebra, is presented by a transversally elliptic operators. So, even if we believe that ``$LT\ltimes_{-\tau} C_0(\ca{M})$'' is truly the crossed product algebra, we only get the following. Therefore, we need more sophisticated ($KK$-theoretical) construction.

\begin{`thm'}
Since $(\ca{H},\ca{D})\in KK(LT\ltimes_{-\tau} C_0(\ca{M}),\bb{C})$, $\ca{D}$ is at least $LT$-transversally elliptic operator on $\ca{M}$.
\end{`thm'}


To overcome such problem, it must be the best to define another $C^*$-algebra which can plays a role of $C_0(\ca{M})$, and study the assembly map. But such an algebra has never constructed, and it seems too difficult.
As the (probably) second best, we will construct the ``latter half'' of the assembly map, and compare it with the analytic index.



\begin{thm}[Main result]\label{Main Result}
Let $\ca{M}$, $\ca{L}$, $\ca{S}$, $\ca{H}$ and $\ca{D}$ be as in Theorem 1.2. We regard the pair $x:=(\ca{H},\ca{D})$ as a ``virtual $\tau$-twisted $LT$-equivariant $K$-homology class of $KK_{LT}^\tau(C_0(\ca{M}),\bb{C})$''.
We can construct the followings:
\begin{itemize}
\item a $C^*$-algebra which can be regarded as $LT\ltimes C_0(\ca{M})$,
\item an element ``$j^{LT}_\tau(x)$''$\in KK(LT\ltimes C_0(\ca{M}),LT\ltimes_\tau \bb{C})$ directly, without the $K$-homology group $KK_{LT}^\tau(C_0(\ca{M}),\bb{C})$, and
\item a $K$-theory class $[c]\in KK(\bb{C},LT\ltimes C_0(\ca{M}))$ which plays a role of the Mishchenko line bundle.
\end{itemize}
These objects enable us to define a $KK$-theoretical index $\mu^{LT}_\tau(x)$.

Moreover, we can define an analytic index $\ind_{LT\ltimes_\tau\bb{C}}(x)\in KK(\bb{C},LT\ltimes_\tau \bb{C})$, and it coincides with the $KK$-theoretical index $\mu^{LT}_\tau(x)$.

\end{thm}

\begin{rmk}
We summarize here what will be (or has been) truly constructed, and what has been only ``virtually constructed''.
For example, the pair $(\ca{H},\ca{D})$ has been truly constructed, and we regard it as an ``element of $KK_{LT}^\tau(C_0(\ca{M}),\bb{C})$''. But the function algebra has not been defined, hence the $K$-homology group is virtual. We use several notations which will be defined later.

{\bf Truly constructed}: $L^2(\ca{M},\ca{S})$, $\ca{D}$, $LT\ltimes_{\pm\tau} \bb{C}$, $LT\ltimes C_0(\ca{M})$, $[c]\in KK(\bb{C},LT\ltimes C_0(\ca{M}))$, $j^{LT}_\tau(x)\in KK(LT\ltimes C_0(\ca{M}),LT\ltimes_\tau \bb{C})$ and $\ind_{LT\ltimes_\tau\bb{C}}(x)\in KK(\bb{C},LT\ltimes_\tau \bb{C})$.

{\bf Virtually constructed}: $C_0(\ca{M})$, $j^{LT}_\tau:KK(C_0(\ca{M}),\bb{C})\to KK(LT\ltimes C_0(\ca{M}),LT\ltimes_\tau\bb{C})$, a cut-off function $c: \ca{M}\to \bb{C}$, and related objects including $KK_{LT}^\tau(C_0(\ca{M}),\bb{C})$.

\end{rmk}

This paper is organized as follows.

In Section 2, we will prepare several $KK$-theoretical matters. In particular, we study the assembly map for a $U(1)$-central extension group. For a group $\Gamma$, a central extension group $\Gamma^\tau$ itself is also a group, hence we can apply the assembly map construction for the whole group $\Gamma^\tau$. However, if a $K$-homology cycle satisfies a certain condition, we do not need to deal with the whole information of the crossed product, and the group $C^*$-algebra. This observation seems quite natural and not highly non-trivial, but it will play an essential role in Section 4. 

In Section 3, as a model of infinite-dimensional case, we will explicitly describe the $j$-homomorphism for $G^\tau\curvearrowright G$, and the Mishchenko line bundle, for a locally compact abelian Lie group $G$.

In Section 4, we will study the infinite-dimensional case. Firstly, we will divide the problem into two parts, just as in \cite{T}. After the brief review of \cite{T}, we will construct several objects, and prove the main result.

\begin{rmks}\label{ass. of grp}
$(1)$ To avoid annoying and non-essential problems, we will deal with locally compact, amenable and unimodular groups, except for $LT$ and related groups. Throughout this paper, $\Gamma$, $ \Gamma_1$ and $\Gamma_2$ are locally compact, amenable and unimodular groups, and $G$ is a locally compact, amenable and unimodular Lie group. So we do not distinguish full group $C^*$-algebras from reduced ones, and the Haar measures are always two-sided invariant. Let us notice that a $U(1)$-central extension of an amenable group is also amenable (\cite{Pie}).

$(2)$ We will use the graded language: $[D_1,D_2]:=D_1D_2-(-1)^{\deg(D_1)\deg(D_2)}D_2D_1$, and $\id\otimes D(u\otimes v)=(-1)^{\deg(D)}u\otimes Dv$.

\end{rmks}

\section{Preliminaries}

In this section, we prepare a better description of the assembly map for a $U(1)$-central extension group. For this aim, we recall several results about unbounded Kasparov modules. We also study the assembly map in the unbounded picture.

\subsection{Unbounded Kasparov modules}

We begin with a review of $KK$-theory in the unbounded picture introduced in \cite{BJ}. For simplicity, we assume that $A$ and $B$ are trivially graded.

\begin{dfn}
For $C^*$-algebras $A$ and $B$, an unbounded Kasparov $A$-$B$-module is a pair $(E,D)$ such that:
\begin{itemize}
\item $E$ is a $\bb{Z}_2$-graded countably generated Hilbert $B$-module equipped with an even $*$-homomorphism $A\to \ca{L}_B(E)$.  Using this homomorphism, we regard $E$ as an $A$-$B$-bimodule.
\item $D:\dom(D)\to E$ is a closed, self-adjoint, regular and adjointable $B$-module homomorphism (see below).
\item There exists a dense $*$-subalgebra $A_1$, for any $a\in A_1$, $a$ preserves $\dom(D)$, $[D,a]$ defines an element $\ca{L}_B(E)$, and $(1+D^2)^{-1}a$ defines an element $\ca{K}_B(E)$.
\end{itemize}
The set of homotopy classes of unbounded Kasparov $A$-$B$-modules is written as $\Psi(A,B)$ which is an abelian group with respect to the direct sum.

For $\Gamma$-$C^*$-algebras $A$ and $B$, an unbounded $\Gamma$-equivariant Kasparov $A$-$B$-module $(E,D)$ is a Kasparov $A$-$B$-module such that $\Gamma$ acts on $E$ satisfying that $\gamma.(aeb)=(\gamma.a)\cdot(\gamma.e)\cdot(\gamma.b)$ for $a\in A$, $e\in E$ and $b\in B$, and $D$ is $\Gamma$-equivariant. The set of homotopy classes of $\Gamma$-equivariant Kasparov $A$-$B$-modules is written as $\Psi_\Gamma(A,B)$.

\end{dfn}

\begin{rmks}
$(1)$ An operator $D:\dom(D)\to E$ is said to be regular, if $1+D^*D$ has dense range (\cite{Kuc}). Unlike the Hilbert space case, this condition is not automatic.

$(2)$ An adjointable operator is automatically $B$-linear and closable. But just like the Hilbert space case, self-adjointness is not automatic. In fact, as observed in \cite{Kuc}, if $D\pm i$ have dense range, $D$ is self-adjoint and regular.
\end{rmks}

\begin{fact}[\cite{BJ}]
Let $b:\bb{R}\to\bb{R}$ be a function defined by $b(x):=x/\sqrt{1+x^2}$. The correspondence $(E,D)\mapsto (E,b(D))$ defines a map $\fra{b}:\Psi(A,B)\to KK(A,B)$. 

The same formula defines a map $\Psi_\Gamma(A,B)\to KK_\Gamma(A,B)$.
\end{fact}

One of the merits to study unbounded Kasparov modules is the simple formula for the exterior tensor product. Let $\otimes_\bb{C}$ denote the exterior tensor product in $KK$-groups.

\begin{pro}\label{BaajJulg}
For $(E_1,D_1)\in\Psi(A,B)$ and $(E_2,D_2)\in \Psi(C,D)$, let $(E_1,D_1)\overline{
\otimes_\bb{C}} (E_2,D_2):=(E_1\otimes E_2,D_1\otimes\id+\id\otimes D_2)$. Then
$$\fra{b}(E_1,D_1)\otimes_\bb{C} \fra{b}(E_2,D_2)= \fra{b}\bigl((E_1,D_1)\overline{
\otimes_\bb{C}} (E_2,D_2)\bigr)$$
in $KK(A\otimes C,B\otimes D)$. Notice that $\id\otimes D_2$ is the graded tensor product.
\end{pro}

For this reason, we wright $\overline{\otimes_\bb{C}}$ as $\otimes_\bb{C}$ from now on.

The heart of $KK$-theory is the Kasparov product (\cite{JT} and \cite{Bla}). An unbounded version of the criterion to judge if a $KK$-cycle is a product of two other cycles, was studied in \cite{Kuc}.
The following is a weaker version of his result, which is enough for our purpose.

\begin{pro}[\cite{Kuc}]\label{Kuc}
An unbounded Kasparov module $(E_1\otimes_B E_2,D)\in \Psi(A,C)$ represents the Kasparov product of $(E_1,D_1)\in \Psi(A,B)$ and $(E_2,D_2)\in\Psi(B,C)$, if the following conditions are satisfied:
\begin{itemize}
\item for all $x$ in some dense sebset of $E_1$, the graded commutator
$$\bbbra{\begin{pmatrix} D & 0 \\ 0 & D_2 \end{pmatrix}, \begin{pmatrix} 0 & T_{e_1} \\ T_{e_1}^* & 0 \end{pmatrix}}$$
is bounded on $\dom(D)\oplus \dom(D_2)$, where $T_{e_1}:E_2\to E_1\otimes_B E_2$ is defined by $T_{e_1}(e_2):=e_1\otimes_B e_2$,
\item $\dom(D)\subseteq \dom(D_1)\otimes_B E_2$, and
\item $\innpro{D_1\otimes\id(e)}{D(e)}{}+ \innpro{D(e)}{D_1\otimes\id(e)}{}\geq 0$ for all $e\in\dom(D)$.
\end{itemize}

\end{pro}





\subsection{Assembly maps for unbounded cycles}


We have two tasks in this subsection:
\begin{itemize}
\item to rewrite the $j$-homomorphism in the unbounded picture, and
\item deduce the product formula for the assembly map.
\end{itemize}





To rewrite the assembly map for unbounded Kasparov modules, we prepare the $j$-homomorphism for the unbounded model. Let $A$ and $B$ be $\Gamma$-$C^*$-algebras. The $j$-homomorphism $j^\Gamma:KK_\Gamma(A,B)\to KK(\Gamma\ltimes A,\Gamma\ltimes B)$ is defined as follows, in the bounded picture.

\begin{dfn}
Let $A$ and $B$ be $\Gamma$-$C^*$-algebras.
Let $(E,F)$ be a $\Gamma$-equivariant Kasparov $A$-$B$-module.
$C_c(\Gamma,E)$ has a $C_c(\Gamma,A)$-$C_c(\Gamma,B)$-bimodule structure as follows:
$$a*e(\gamma):=\int_\Gamma a(\gamma')\gamma'.\bra{e(\gamma'^{-1}\gamma)}d\gamma',$$
$$e*b(\gamma):=\int_\Gamma e(\gamma')\gamma'.\bra{b(\gamma'^{-1}\gamma)}d\gamma',$$
$$\innpro{e_1}{e_2}{\Gamma\ltimes B}(\gamma):=\int_\Gamma \innpro{e_1(\gamma')}{e_2(\gamma'\gamma)}{B}d\gamma'.$$
The $(\Gamma\ltimes A)$-$(\Gamma\ltimes B)$-bimodule $\Gamma\ltimes E$ is given by the completion with respect to $\innpro{\cdot}{\cdot}{\Gamma\ltimes B}$.

$\widetilde{F}(e)(\gamma)$ is defined by $F\bra{e\bra{\gamma}}$.

$j^\Gamma(E,F)$ is given by $(\Gamma\ltimes E,\widetilde{F})\in KK(\Gamma\ltimes A,\Gamma\ltimes B)$.
\end{dfn}

The unbounded version of this map can be described as follows.
\begin{lem}
For $(E,D)\in \Psi_\Gamma(A,B)$, $\fra{b}(\Gamma\ltimes E,\widetilde{D})=j^\Gamma(\fra{b}(E,D))$, where $\widetilde{D}e(\gamma)=D\bbbra{e(\gamma)}$ for $e\in C_c(\Gamma,E)$.
\end{lem}
\begin{proof}
If we verify that $(\Gamma\ltimes E,\widetilde{D})$ defines an unbounded Kasparov module, the result is clear from the formula $\widetilde{D_1D_2}=\widetilde{D_1}\widetilde{D_2}$, that is, $\widetilde{b(D)}=b(\widetilde{D})$.\\\\
%
%
%
\underline{$\widetilde{D}$ is closed, self-adjoint and regular:} 
We only check that $\widetilde{D}$ is regular, that is, $\id+\widetilde{D}^2$ has dense range. Let us notice that the restriction of $\widetilde{D}$ to $\dom(D)\otimes C_c(\Gamma)$ can be written as $D\otimes \id$, and $(\id+D^2)\otimes\id|_{(\dom(D^2))\otimes C_c(\Gamma)}$ has dense range in $\Gamma\ltimes E$, by a usual argument to approximate general functions by $C_c(\Gamma)$. Clearly the range of $\id+\widetilde{D}^2$ contains the range of $(\id+ D^2)\otimes\id$, hence $\widetilde{D}$ is regular.\\\\
\underline{There is a dense subalgebra $(\Gamma\ltimes A)_1$:} Let $A_1$ be a dense subalgebra of $A$ such that $[D,a]$ is bounded. The subalgebra $A_1\otimes C_c(\Gamma)\subseteq \Gamma\ltimes A$ satisfies the required condition: for $a\otimes f\in A_1\otimes C_c(\Gamma)$ and $e\in C_c(\Gamma,\dom(D))$,
\begin{align*}
[\widetilde{D},a\otimes f]e(\gamma)&=  [D,a] \int f(\gamma')\gamma'.\bigl(e(\gamma'^{-1}\gamma)\bigr)d\gamma',
\end{align*}
which is bounded with respect to $e$.\\\\
\underline{$(1+\widetilde{D}^2)^{-1}a$ is compact:} It is enough to verify the statement when $a$ is an element of $A\otimes C_c(\Gamma)$. In this case, the element $(1+\widetilde{D}^2)^{-1}a$ belongs to $C_c(\Gamma,\ca{K}_B(E))\subseteq \ca{K}_{\Gamma\ltimes B}(\Gamma\ltimes E)$.

\end{proof}

We use the same symbol $j^\Gamma$ for the corresponding map $\Psi_\Gamma(A,B)\to \Psi(\Gamma\ltimes A,\Gamma\ltimes B)$ given by the correspondence $(E,D)\mapsto (\Gamma\ltimes E,\widetilde{D})$.


\vspace{0.2cm}
Let us move to the second task. 

The $j$-homomorphism behaves well, under the tensor product.
\begin{lem}
Let $A_j$ and $B_j$ be $\Gamma_j$-$C^*$-algebras, for $j=1,2$. Under the assumption of Lemma 3.6 in \cite{T},
the following diagram commutes:
{\small 
$$\begin{CD}
\Psi_{\Gamma_1}(A_1,B_1)\otimes \Psi_{\Gamma_2}(A_2,B_2) @>\otimes_\bb{C}>> \Psi_{\Gamma_1\times \Gamma_2}(A_1\otimes A_2,B_1\otimes B_2) \\
@V{j^{\Gamma_1}\otimes j^{\Gamma_2}}VV @VVj^{\Gamma_1\times \Gamma_2}V \\
\Psi(\Gamma_1\ltimes A_1, \Gamma_1\ltimes B_1)\otimes \Psi(\Gamma_2\ltimes A_2, \Gamma_2\ltimes B_2) @>\otimes_\bb{C}>> \Psi((\Gamma_1\times \Gamma_2)\ltimes (A_1\otimes A_2), (\Gamma_1\times \Gamma_2)\ltimes (B_1\otimes B_2)).
\end{CD}$$
}

\end{lem}
\begin{proof}
It is almost clear from that $(\Gamma_1\ltimes E_1)\otimes (\Gamma_2\ltimes E_2)\cong (\Gamma_1\times \Gamma_2)\ltimes (E_1\otimes E_2)$ and Proposition \ref{BaajJulg}.

\end{proof}

The next step is to study Mishchenko line bundles. Let us recall the definition of them.

\begin{dfn}
Let $X$ be a complete Riemannian manifold equipped with a proper and cocompact action of $\Gamma$. Then there exists a smooth and compactly supported function $c:X\to \bb{R}_{\geq 0}$ such that
$\int_\Gamma c(\gamma.x)d\gamma=1$ for any $x\in X$. It defines an idempotent $[c]\in \Gamma\ltimes C_0(X)$ by the formula $\{[c](\gamma)\}(x)=\sqrt{c(x)c(\gamma^{-1}.x)}$.


In the $KK$-theoretical language, $[c]$ is given by $([c]*\bigl(\Gamma\ltimes C_0(X)\bigr),0)\in \Psi(\bb{C},\Gamma\ltimes C_0(X))$. This $K$-theory class is called the Mishchenko line bundle associated to $\Gamma\curvearrowright X$.

\end{dfn}

We can factorize the Mishchenko line bundle for $\Gamma_1\times \Gamma_2\curvearrowright X_1\times X_2$ as follows:

\begin{lem}
Let $c_1:X_1\to \bb{R}_{\geq0}$ and $c_2:X_2\to \bb{R}_{\geq0}$ be chosen cut-off functions. Then $c_1\otimes c_2:X_1\times X_2\to \bb{R}_{\geq0}$ defined by $c_1\otimes c_2(x_1,x_2)=c_1(x_1)c_2(x_2)$ is a cut-off function associated to $(\Gamma_1\times \Gamma_2)\curvearrowright (X_1\times X_2)$. 

Moreover, $[c_1\otimes c_2]=[c_1]\otimes_\bb{C} [c_2]\in KK(\bb{C},(\Gamma_1\times\Gamma_2)\ltimes C_0(X_1\times X_2))$ under the identification $\bigl(\Gamma_1\ltimes C_0(X_1)\bigr)\bigotimes \bigl(\Gamma_2\ltimes C_0(X_2)\bigr)\cong (\Gamma_1\times \Gamma_2)\ltimes C_0(X_1\times X_2)$.
\end{lem}

Combining them, we get a useful formula.

\begin{pro}\label{product formula}
The following diagram commutes.
{\small 
$$
\begin{CD}
\Psi_{\Gamma_1}(C_0(X_1),\bb{C})\otimes \Psi_{\Gamma_2}(C_0(X_2),\bb{C})
@>\otimes_\bb{C}>> \Psi_{\Gamma_1\times \Gamma_2}(C_0(X_1\times X_2),\bb{C}) \\
@ Vj^{\Gamma_1}\otimes j^{\Gamma_2}VV @Vj^{\Gamma_1\times \Gamma_2}VV \\
\Psi(\Gamma_1\ltimes C_0(X_1),\Gamma_1\ltimes\bb{C})\otimes \Psi(\Gamma_2\ltimes C_0(X_2),\Gamma_2\ltimes\bb{C}) @>\otimes_\bb{C}>> \Psi((\Gamma_1\times \Gamma_2)\ltimes C_0(X_1\times X_2),(\Gamma_1\times \Gamma_2)\ltimes \bb{C}) \\
@V[c_1]\otimes_{\Gamma_1\ltimes C_0(X_1)}\bigotimes [c_2]\otimes_{\Gamma_2\ltimes C_0(X_2)}VV @V([c_1]\otimes_\bb{C}[c_2])\bigotimes_{(\Gamma_1\times \Gamma_2)\ltimes C_0(X_1\times X_2)}VV \\
\Psi(\bb{C},\Gamma_1\ltimes\bb{C})\otimes KK(\bb{C},\Gamma_2\ltimes\bb{C}) @>>\otimes_\bb{C}> \Psi(\bb{C},\Gamma_1\ltimes\bb{C}\otimes \Gamma_2\ltimes\bb{C}).
\end{CD}
$$
}
\end{pro}



\subsection{The assembly map for a central extension of a group}

As a model of the main construction, we study the case of finite-dimensional manifolds. Let $X$ be a finite-dimensional, complete, Riemaniann manifold, equipped with a proper, isometric, cocompact action of $\Gamma$. We deal with a $U(1)$-central extension group 
$$1\to U(1)\xrightarrow{i} \Gamma^\tau\xrightarrow{p} \Gamma\to 1$$
and a Clifford bundle $\pi:W\to X$ equipped with a $\Gamma^\tau$-action satisfying that $\pi(g.w)=p(g).\pi(w)$, and $i(z).w=z^kw$ for $z\in U(1)$. 
In fact it suffices to study the case when $k=1$, and we will assume that later.
Assume that $D$ is a $\Gamma^\tau$-equivariant Dirac operator acting on $L^2(X,W)$.
In this situation, we can define an unbounded Kasparov module $(L^2(X,W),D)\in\Psi_{\Gamma^\tau}(C_0(X),\bb{C})$.
The goal of this subsection is a simpler description of $\mu^{\Gamma^\tau}([(L^2(X,W),D)])\in \Psi(\bb{C},\Gamma^\tau\ltimes\bb{C})$, using subalgebras of crossed product algebras. The following is an obvious generalization of the above $KK$-cycle.

\begin{dfn}
Let $A$ and $B$ be $\Gamma$-$C^*$-algebras. 
Through the homomorphism $p:\Gamma^\tau\to \Gamma$, $A$ and $B$ happen to be $\Gamma^\tau$-$C^*$-algebras.
An unboudned $KK$-cycle $(E,D)\in \Psi_{\Gamma^\tau}(A,B)$ is {\bf at level $k$} if the $\Gamma^\tau$-action satisfies that $i(z).e=z^ke$ for $z\in U(1)$. The set of homotopy classes of $KK$-cycles at level $k$ is denoted by $\Psi_\Gamma^{k\tau}(A,B)$.
Let $KK_\Gamma^{k\tau}(A,B)$ be the bounded version of $\Psi_\Gamma^{k\tau}(A,B)$.

\end{dfn}

Since $\Psi_\Gamma^{k\tau}(A,B)$ is much smaller than the whole $KK$-group $\Psi_{\Gamma^\tau}(A,B)$, it is natural to expect the restriction of the assembly map to $\Psi_\Gamma^{k\tau}(C_0(X),\bb{C})$ factors through a certain smaller group. To realize such an idea, we prepare several subalgebras and submodules. 





\begin{dfn}
For an integer $m$, let
$$\Gamma\ltimes_{m\tau}A:=\overline{\bbra{f:\Gamma^\tau\to A\mid f(zg)=z^mf(g)}}$$
be a $C^*$-algebra (verified later) obtained by the closure in $\Gamma^\tau\ltimes A$, and let
$$\Gamma\ltimes_{m\tau} E:=\overline{\bbra{\phi:\Gamma^\tau\to E\mid \phi(zg)=z^m\phi(g)}}$$
be a $(\Gamma\ltimes_{(m-k)\tau}A)$-$(\Gamma\ltimes_{m\tau}B)$-bimodule (verified later) obtained by the closure in $\Gamma^\tau\ltimes E$.
Such functions are said to be {\bf at level $m$}.
\end{dfn}

Let us check that $\Gamma\ltimes_{m\tau}A$ is a $C^*$-algebra.

\begin{lem}
$\Gamma\ltimes_{n\tau}A$ is a $C^*$-subalgebra in $\Gamma^\tau\ltimes A$.
\end{lem}
\begin{proof}
It is enough to check that the set of functions from $\Gamma^\tau$ to $A$ at level $n$ is a $*$-subalgebra.
Let $a$, $a':\Gamma^\tau\to A$ be at level $n$. Clearly $a+a'$ and $\lambda a$ are also at level $n$, for any $\lambda \in\bb{C}$. Let us check that $a*a'$ is also:
\begin{align*}
a*a'(zg)&= \int a(h)h.\bra{a'\bra{h^{-1}zg}} dh\\
&=z^na*a'(g),
\end{align*}
since $i(U(1))$ acts on $A$ trivially.
Hence $\Gamma\ltimes_{n\tau}A$ is a subalgebra. To check the $*$-closedness,
$$a^*(zg)=\overline{a(z^{-1}g^{-1})}=z^n\overline{a(g^{-1})}=z^na^*(g).$$
\end{proof}

In fact, $\Gamma^\tau\ltimes A$ can be decomposed as $C^*$-algebras. The following can be verified by the same technique of Lemma \ref{convolution matter}.
\begin{lem}
$(\Gamma\ltimes_{m_1\tau}A) * (\Gamma\ltimes_{m_2\tau}A)=0$ unless $m_1= m_2$.
\end{lem}

To verify that $\Gamma\ltimes_{m\tau}E$ is a $(\Gamma\ltimes_{(m-k)\tau}A)$-$(\Gamma\ltimes_{m\tau}B)$-bimodule, we check the following lemma.

\begin{lem}\label{convolution matter}
$(1)$ $(\Gamma\ltimes_{m\tau}A)\cdot(\Gamma\ltimes_{n\tau}E)=0$ unless $m=n-k$.

$(2)$ $(\Gamma\ltimes_{m\tau}E)\cdot (\Gamma\ltimes_{n\tau}B)=0$ unless $n=m$.

$(3)$ $\innpro{\Gamma\ltimes_{m\tau}E}{\Gamma\ltimes_{n\tau}E}{}\subseteq\begin{cases} \Gamma\ltimes_{m\tau}B & (m=n) \\ 0 & (m\neq n).\end{cases}$

\end{lem}
\begin{proof}
$(1)$ Let $a:\Gamma^\tau\to A$ be at level $m$, and $\phi:\Gamma^\tau\to E$ be at level $n$.
$$a*\phi(zg) = \int a(h)h.\bra{\phi\bra{h^{-1}zg}}dh= z^na*\phi(g),$$
hence at level $n$, and by the change of variables $h=zu$,
$$a*\phi(zg) = \int a(h)h.\bra{\phi\bra{h^{-1}zg}}dh$$
$$= \int a(zu)(zu).\bra{\phi\bra{u^{-1}g}}du=z^{m+k}a*\phi(g),$$
hence at level $m+k$, at the same time. Therefore such a function must vanish unless $m=n-k$.

$(2)$ Let $b:\Gamma^\tau\to B$ be at level $n$, and $\phi:\Gamma^\tau\to E$ be at level $m$.
Similarly, but more easily,
$\phi*b$ is at level $n$, and at the same time at level $m$. Such a function must vanish unless $n=m$.

$(3)$ Let $e_n$, $e_m$ be at level $n$, $m$ respectively.
By the same technique, $\innpro{e_n}{e_m}{\Gamma^\tau\ltimes B}$ is at level $m$, and at the same time at level $n$.


\end{proof}

The following is clear from the definition of $\widetilde{D}$.
\begin{lem}
$\widetilde{D}$ preserves the decomposition $\bigoplus \Gamma\ltimes_{n\tau}E$.
\end{lem}

Combining all of them, we have reached the definition of the ``partial $j$-homomorphism'' as follows. We use the same symbol for the restriction of $\widetilde{D}$ to each submodules.


\begin{dfn}
The partial $j$-homomorphism at level $m$
$$j^\Gamma_{m\tau}:\Psi_\Gamma^{k\tau}(A,B)\to \Psi(\Gamma\ltimes_{(m-k)\tau}A,\Gamma\ltimes_{m\tau}B)$$
is given by the following objects:
\begin{itemize}
\item $\Gamma\ltimes_{m\tau}E$ as the 
$(\Gamma\ltimes_{(m-k)\tau}A)$-$(\Gamma\ltimes_{m\tau}B)$-module,
\item $\widetilde{F}|_{\Gamma\ltimes_{m\tau}E}\in \ca{L}_{\Gamma\ltimes_{m\tau}B}(\Gamma\ltimes_{m\tau}E)$ as the operator.
\end{itemize}

\end{dfn}



\begin{rmks}
$(1)$ In the following, we will deal with only the case when $k=1$, and the map we need later is only $j^\Gamma_\tau:\Psi_\Gamma^\tau(A,B)\to \Psi(\Gamma\ltimes A,\Gamma\ltimes_\tau B)$. The same formula defines a bounded version $j^\Gamma_\tau:KK_\Gamma^\tau(A,B)\to KK(\Gamma\ltimes A,\Gamma\ltimes_\tau B)$.

$(2)$ The product formula is valid also for the partial $j$-homomorphisms.

\end{rmks}

Let us move to the study of the Mishchenko line bundles.
We have supposed that $\Gamma^\tau$ acts on $X$, and the restriction to $i(U(1))$  is trivial. This condition implies that the function $[c]:\Gamma^\tau\to C_0(X)$ is invariant under the $i(U(1))$-action, that is, $[c]$ is at level $0$. 
Hence $[c]*(\Gamma\ltimes_{l\tau}C_0(X))$ vanishes unless $l=0$. So we can regard $[c]$ as belonging to the direct summand $\Psi(\bb{C},\Gamma\ltimes C_0(X))\subseteq \Psi(\bb{C},\Gamma^\tau\ltimes C_0(X))$.
A similar computation of $(1)$ in Lemma \ref{convolution matter} implies that: if $(E,D)\in \Psi(\Gamma^\tau\ltimes C_0(X),C)$ is at level $k\neq 0$, $[c]\otimes_{\Gamma^\tau \ltimes C_0(X)} (E,D)=0$, for any $C^*$-algebra $C$.

\begin{pro}
The restriction of the assembly map can be rewritten by the simplified map:

$$\begin{xymatrix}{
\Psi_\Gamma^{k\tau}(C_0(X),\bb{C}) \ar@/_60pt/[dd]_{\mu^\Gamma_\tau}
\ar[d]^{j^\Gamma_{k\tau}} \ar@{^{(}-_>}[r]& \Psi_{\Gamma^\tau}(C_0(X),\bb{C}) \ar[d]_{j^{\Gamma^\tau}}  \ar@/^60pt/[dd]^{\mu^{\Gamma^\tau}}
\\
\Psi(\Gamma\ltimes C_0(X),\Gamma\ltimes_{k\tau}\bb{C}) 
\ar[d]^{[c]\otimes_{\Gamma\ltimes C_0(X)}-} & \Psi(\Gamma^\tau\ltimes C_0(X),\Gamma^\tau\ltimes\bb{C}) \ar[d]_{[c]\otimes_{\Gamma^\tau\ltimes C_0(X)-}}\\
\Psi(\bb{C},\Gamma\ltimes_{k\tau}\bb{C}) \ar@{^{(}-_>}[r] & \Psi(\bb{C},\Gamma^\tau\ltimes\bb{C}).
}\end{xymatrix}$$


\end{pro}

\begin{rmk}
This theorem does not seem to be highly non-trivial. Thanks to this theorem, however, we can avoid the construction of the whole group $C^*$-algebra of the central extension group. In fact, the ordinary group $C^*$-algebra $LT\ltimes\bb{C}$ has not been defined, nevertheless we can define a $C^*$-algebra which plays a role of twisted group $C^*$-algebra $LT\ltimes_\tau \bb{C}$ of $LT$ at level $\tau$.
\end{rmk}


\section{Finite-dimensional case}
As we stated above, we compute the formula we need in the locally compact setting, then we define the corresponding object for $\Omega T_0$, imitating the formula in this section. Before that, we describe the $K$-homology class $(L^2(G,\tau)\otimes S_G,\Dirac_R)\in \Psi^\tau_G(C_0(G),\bb{C})$ to study the index problem.

\subsection{The $K$-homology class $(L^2(G,\tau)\otimes S_G,\Dirac_R)\in \Psi_G^\tau(C_0(G),\bb{C})$}

In order to get a $\tau$-twisted $G$ representation as an analytic index, we need to consider a $\tau$-twisted $G$-equivariant Spinor bundle. For simplicity, we deal with a $\tau$-twisted $G$-equivariant Spinor bundle of the form
\begin{center}
$\tau$-twisted $G$-equivariant line bundle $\otimes$ $G$-equivariant Spinor bundle.
\end{center}
We focus on the following line bundle.

\begin{dfn}
Let $\ca{L}\to G$ be the line bundle $\ca{L}:=G^\tau\times_{U(1)}\bb{C}$.
\end{dfn}

\begin{lem}
There is a one to one correspondence between a section of $\ca{L}$ and a function on $G^\tau$ at level $-1$. Thinking over this fact, $L^2(G,\tau)$ denotes the Hilbert space consisting of $L^2$ sections of $\ca{L}$.

\end{lem}
\begin{proof}
Let $f:G^\tau\to \bb{C}$ be at level $-1$. We can define a section $s$ by the formula
$$s(x):=[q,f(q)]\in G^\tau\times_{U(1)}\bb{C}.$$
Since $f$ is at level $-1$, $[zq,f(zq)]=[zq,z^{-1}f(q)]=[q,f(q)]$, thus $s$ is well-defined. Conversely, for $x\in G$, the value of a section $s$ is written as $[q,w_{q,s(x)}]$, where $q\in G^\tau$ belongs to the fiber at $x$, and $w_{q,s(x)}$ is a complex number depending on the value $s(x)$ and the representative $q$. Let $f:G^\tau\to \bb{C}$ be a function given by $q\mapsto w_{q,s(x)}$. It is clearly at level $-1$.
\end{proof}

The Hilbert space $L^2(G,\tau)$ admits a $\tau$-twisted representation of $G$ at level $1$ by the usual formula: $[g.\phi](x):=\phi(g^{-1}x)$. Moreover, $L^2(G,\tau)$ admits a $*$-representation of $C_0(G)$ by the left multiplication, where we regard $C_0(G)$ as the subalgebra of $C_0(G^\tau)$ at level $0$, using the pullback $p^*$.
These actions are compatible in the following sense:
$$[g.(f\cdot\phi)](x)=(f\cdot\phi)(g^{-1}x)=f(g^{-1}x)\phi(g^{-1}x)=[(g.f)\cdot(g.\phi)](x),$$
for $g$, $x\in G^\tau$, $f\in C_0(G)$ and $\phi\in L^2(G,\tau)$.
This construction extends to $L^2(G,\tau)\otimes S$ by the obvious way.

For simplicity, we assume that {\bf $G$ is abelian} for the reason in Remark \ref{excuse}.
If we fix a splitting $s:\fra{g}\to \fra{g}^\tau$ of $dp:\fra{g}^\tau\to\fra{g}$, and take a orthonomal basis $\{e_i\}$ of $\fra{g}$ with respect to a $G^\tau$-invariant inner product, we can define a Dirac operator acting on $L^2(G,\tau)\otimes S$ by the formula
$$\cancel{\partial}_R:=\sum_idR_{s(e_i)}\otimes\gamma(e_i),$$
where $R$ is the right regular representation of $G^\tau$, $dR$ is its infinitesimal version, and $\gamma:\fra{g}\to \End(S)$ is the Clifford multiplication. Since $\cancel{\partial}_R$ is $G^\tau$-invariant with respect to the left regular representation, we have got the following.

\begin{pro}
$x:=(L^2(G,\tau)\otimes S,\cancel{\partial}_R)$ defines an element of $\Psi_G^\tau(C_0(G),\bb{C}).$
\end{pro}

\begin{rmk}\label{excuse}
For a non-abelian group $G$, it is natural to add the potential term coming from the Spin representation.
Let us recall that the right action on $TG\cong G\times\fra{g}$ is the adjoint action if we trivialize $TG$ by the left translation. Therefore if we define the Spinor bundle $S$ by the left translation, the right action on the Spinor bundle must be the lift of the adjoint representation, that is, the Spin representation.
More generally, we need to consider the Spin extension $U(1)\to Ad^*(Spin^c(\fra{g})\to SO(\fra{g}))\to G$ of $G$ to deal with the Spin representation. To avoid such problems which are not essential for us ($LT$ is abelian!), we assume that $G$ is abelian.

\end{rmk}

\subsection{The value of the $j$-homomorphism}

We compute $j^G_\tau(x)\in \Psi(G\ltimes C_0(G),G\ltimes_\tau \bb{C})$ explicitly in this subsection.

Since $C_0(G)$ and $G^\tau$ acts on $S$-part trivially, it suffices to study $G\ltimes_\tau L^2(G,\tau)$ in order to compute the module $G\ltimes_\tau L^2(G,\tau)\otimes S$. Notice that $G\ltimes C_0(G)\cong \ca{K}(L^2(G))$.
\begin{pro}\label{j-hom for locally compact}
$G\ltimes_\tau L^2(G,\tau)$ is isomorphic to $L^2(G)\otimes (G\ltimes_\tau \bb{C})$ as $(G\ltimes C_0(G))$-$(G\ltimes_\tau \bb{C})$-bimodules.
\end{pro}
\begin{proof}
Let $C_c(G,\tau)$ be the set of compactly supported continuous functions of $G^\tau$ at level $1$.
Let us consider a map $m:L^2(G)\otimes C_c(G,\tau)\to C_c(G^\tau,L^2(G,\tau))\subseteq G\ltimes_\tau L^2(G,\tau)$ defined by
$$m(\phi_1\otimes\phi_2)(x,g):=\phi_1(x)\phi_2(x^{-1}g).$$
More precisely, if we are given $g\in G^\tau$, we define an element of $L^2(G,\tau)$ by the formula $\phi_1(\bullet)\phi_2(\bullet^{-1}g)$. It is truly $L^2$, since $\phi_1\in L^2$ and $\phi_2\in C_c$, it is truly at level $1$ with respect to $g$, since $\phi_2$ is at level $1$, and it is truly at level $-1$ with respect to $\bullet$, since $\phi_1$ is at level $0$ and $\phi_2$ is at level $1$.
Let us verify that $m$ is a $(G\ltimes C_0(G))$-$(G\ltimes_\tau \bb{C})$-bimodule isomorphism.\\\\
\underline{$G\ltimes C_0(G)$-homomorphism:} For a compactly supported $a:G\to C_0(G)$,
\begin{align*}
a*m(\phi_1\otimes\phi_2)(x,g) &= \int a(h)h.\bra{m(\phi_1\otimes\phi_2)(\cdot,h^{-1}g)}dh(x) \\
&= \int a(h,x)m(\phi_1\otimes\phi_2)(h^{-1}x,h^{-1}g)dh \\
&= \int a(h,x) \phi_1(h^{-1}x)\phi_2(x^{-1}g)dh\\
&= a*\phi_1(x)\phi_2(x^{-1}g)\\
&= m\bra{\bra{a*\phi_1}\otimes\phi_2}(x,g).
\end{align*}\\
\underline{$G\ltimes_\tau \bb{C}$-homomorphism:} For a compactly supported $b:G^\tau\to \bb{C}$,
\begin{align*}
m(\phi_1\otimes\phi_2)*b(x,g)&= \int \phi_1(x)\phi_2(x^{-1}h)h.(b(h^{-1}g))dh \\
&= \phi_1(x)\int \phi_2(x^{-1}h)b(h^{-1}g)dh \\
&= \phi_1(x) \bra{\phi_2*b}(x^{-1}g)\\
&= m\bra{\phi_1\otimes\bra{\phi_2*b}}(x,g).
\end{align*}\\
\underline{Inner product:} For $\phi_1\otimes\phi_2$, $\psi_1\otimes\psi_2\in L^2(G)\otimes C_c(G,\tau)$, since the $G$-action on $\bb{C}$ is trivial,
\begin{align*}
\innpro{m(\phi_1\otimes\phi_2)}{m(\psi_1\otimes\psi_2)}{G\ltimes_\tau \bb{C}}(g)&=\int\innpro{m(\phi_1\otimes\phi_2)(h)}{m(\psi_1\otimes\psi_2)(hg)}{L^2(G,\tau)}dh\\
&=\int \int\overline{\phi_1(x)\phi_2(x^{-1}h)}\psi_1(x)\psi_2(x^{-1}hg)dxdh\\
&=\int \overline{\phi_1(x)}\psi_1(x)\int\overline{\phi_2(x^{-1}h)}\psi_2(x^{-1}hg)dhdx \\
&= \innpro{\phi_1}{\psi_1}{L^2(G,\tau)}\int \phi_2^*(h^{-1})\psi_2(hg)dh \\
&= \innpro{\phi_1}{\psi_1}{L^2(G,\tau)}\bra{\phi_2^**\psi_2}(g) \\
&=\innpro{\phi_1}{\psi_1}{L^2(G,\tau)}\innpro{\phi_2}{\psi_2}{G\ltimes_\tau \bb{C}}(g),
\end{align*}
where we have used Fubini's theorem, change of variables, right invariance of the measure, and the definition of the involution $*$ in $G\ltimes_\tau \bb{C}$; $f^*(g)=\overline{f(g^{-1})}$.

Therefore $m$ is an isometric embedding. If one notices the image is dense, and $m$ is isometric, the completions $L^2(G)\otimes (G\ltimes_\tau \bb{C})$ and $G\ltimes_\tau L^2(G,\tau)$ are isomorphic to one another.


\end{proof}

For convenience of notation, we regard $G\ltimes_\tau [L^2(G,\tau)\otimes S]$ as $L^2(G)\otimes S\otimes (G\ltimes_\tau \bb{C})$. The next step is to describe $\widetilde{\cancel{\partial}_R}$ in terms of the easier model $L^2(G)\otimes S\otimes (G\ltimes_\tau \bb{C})$. We introduce another operators $D:=\sum dR_{e_i}\otimes\gamma(e_i)$ acting on $C_c^\infty(G)\otimes S\subseteq L^2(G)\otimes S$, and $\cancel{\partial}_L:=\sum_i\gamma(e_i)\otimes dL_{s(e_i)}$ acting on $S\otimes C_c^\infty(G,\tau)\subseteq S\otimes (G\ltimes_\tau \bb{C})$.
\begin{pro}
For $\phi\in C_c^\infty(G)$, $b\in C_c^\infty(G,\tau)$, and $s\in S$,
$$m^{-1}\circ \widetilde{\cancel{\partial}_R}\circ m(\phi\otimes s\otimes b)=
D(\phi\otimes s)\otimes b+\phi\otimes\cancel{\partial}_L(s\otimes b).$$

\end{pro}
\begin{proof}
Noticing the definition of $dL_X$: $dL_Xf(x)=\frac{d}{dt}f(e^{-tX}x)
|_{t=0}$ for $X\in \fra{g}^\tau$, this is just an application of the Leibniz rule.

\end{proof}

As a result, we get the following formula.

\begin{thm}
$j^{G}_\tau(L^2(G,\tau)\otimes S_G,\Dirac_R)$ is represented by
$$(L^2(G)\otimes S_G\otimes (G\ltimes_\tau \bb{C}),D\otimes_2\id+\id\otimes_1\Dirac_L).$$

\end{thm}

\subsection{The Mishchenko line bundle and the value of the assembly map}
We finish this section by describing the value of the assembly map of $x$. For this aim, we rewrite the Mishchenko line bundle under the isomorphism $G\ltimes C_0(G)\cong \ca{K}(L^2(G))$. This observation plays a crucial role of the definition of the Mishcenko line bundle for $\Omega T_0$.

\begin{lem}\label{cut off=proj}
$[c]$ gives the rank one projection onto $\bb{C}\sqrt{c}$ under the isomorphism $G\ltimes C_0(G)\cong \ca{K}(L^2(G))$.
\end{lem}
\begin{proof}
We describe the isomorphism $G\ltimes C_0(G)\to \ca{K}(L^2(G))$ here. Let $a\in C_c(G,C_c(G))$ and $\phi\in L^2(G)$. We suppose that $a$ is of the form  $a(x,g)=a_1(x)\overline{a_2(g^{-1}x)}$, then
\begin{align*}
a*\phi(x) &= \int_G a_1(x)\overline{a_2(h^{-1}x)}\phi(h^{-1}x)dh \\
&=a_1(x)\innpro{a_2}{\phi}{}.
\end{align*}
Hence such $a$ corresponds to the Schatten form $\theta_{a_1,a_2}$.

Since $[c](x,g)$ is define by $\sqrt{c(x)c(g^{-1}x)}$, and it is of the form of the above formula, $[c]$ corresponds to the rank one projection $\theta_{\sqrt{c},\sqrt{c}}$.

\end{proof}

Let us compute $\mu^G_\tau(x):=[c]\otimes_{G\ltimes C_0(G)}j^G_\tau(x)$.
A similar argument will imply the main result in the next section.

\begin{pro}\label{value of the assembly map}
$\mu^G_\tau(x)$ can be represented by $(S\otimes (G\ltimes_\tau \bb{C}),\Dirac_L)$.
\end{pro}
\begin{proof}
In this proof, we identify $G\ltimes C_0(G)$ with $\ca{K}(L^2(G))$, and regard $[c]$ as the rank one projection $P_{\sqrt{c}}$.
In the bimodule language, $[c]$ is represented by $(P_{\sqrt{c}}\circ \ca{K}(L^2(G)),0)\in\Psi(\bb{C},G\ltimes C_0(G))$. 
Any element of the module $P_{\sqrt{c}}\circ \ca{K}(L^2(G))$ can be written as $P_{\sqrt{c}}\circ k=P_{\sqrt{c}}^3\circ k$ for some $k\in \ca{K}(L^2(G))$. Therefore, for any $v\in L^2(G)$,
\begin{align*}
P_{\sqrt{c}}\circ k\otimes_{\ca{K}(L^2(G))}v&=P_{\sqrt{c}}\circ P_{\sqrt{c}}\otimes_{\ca{K}(L^2(G))} P_{\sqrt{c}}\circ k(v)\\
&=P_{\sqrt{c}}\circ P_{\sqrt{c}}\otimes_{\ca{K}(L^2(G))}\sqrt{c}\innpro{\sqrt{c}}{k(v)}{}.
\end{align*}
It tells us that $P_{\sqrt{c}}\circ \ca{K}(L^2(G))\otimes_{\ca{K}(L^2(G))}L^2(G)\cong \bb{C}$ by 
$P_{\sqrt{c}}\circ k \otimes v\mapsto \innpro{\sqrt{c}}{k(v)}{}$, and we identify two modules in the next paragraph.

To check the statement, let us recall Proposition \ref{Kuc}. 
Since $[c]=(P_{\sqrt{c}}\circ\ca{K}(L^2(G)),0)$, the second and the third conditions are satisfied. To verify the first condition, take a compact operator $k$ of the form $\phi\otimes \psi^*$, where $\phi$ and $\psi$ are smooth and compactly supported.
Since $\id\otimes_1 \Dirac_R$ commutes with $k\otimes_1\id$, it essentially suffices to verify that $T_{P_{\sqrt{c}}\circ k}\circ D$ is a bounded operator. 
Since $D=\sum dR_{s(e_i)}\otimes\gamma(e_i)$ is a finite sum, and $\gamma(e_i)$ is bounded, it suffices to check the correspondence $C_c^\infty(G)\ni f\mapsto P_{\sqrt{c}}\circ k\otimes_{\ca{K}(L^2(G))} dR_{s(e_i)}(f)=\innpro{\sqrt{c}}{k(dR_{s(e_i)}(f))}{}\in\bb{C}$ is bounded. It is true because $\innpro{\sqrt{c}}{k(dR_{s(e_i)}(f))}{}=
\innpro{\sqrt{c}}{\phi\innpro{dR_{s(e_i)}^*\psi}{f}{}}{}$, $\psi$ is smooth, and $dR_{s(e_i)}^*$ is also a differential.



\end{proof}
\begin{rmk}\label{aD is bounded}
The above proof essentially contains an important argument: $k\circ dR:L^2\to L^2$ is bounded for ``smooth'' $k$. We will sometimes use this fact.

\end{rmk}

It is not clear that the above index coincides with the analytic index. Let us verify the coincidence of two indices.
\begin{pro}\label{comp of a-ind}
The analytic index is given by $(S_G\otimes (G\ltimes_{-\tau}\bb{C}),\Dirac_R)$. Two indices are isomorphic to one another as Kasparov $\bb{C}$-$(G\ltimes_\tau \bb{C})$-modules.
\end{pro}
\begin{proof}
The first statement is clear from Definition \ref{def of a-ind}. 

For the second one, we give an isomorphism $\Phi:G\ltimes_{-\tau}\bb{C}\to G\ltimes_\tau \bb{C}$ as $G\ltimes_\tau\bb{C}$-Hilbert modules:
$\Phi(f)(g):=f^{\lor}(g):=f(g^{-1}).$ One can verify that the $\Phi$ gives a Hilbert module isomorphism, if he notices that:
$$f*b(x):=\int_G f(g^{-1}.x)b(g^{-1})dg,$$
$$\innpro{f_1}{f_2}{}(g):=\int_G\overline{f_1(x)}f_2(g^{-1}x)dx$$
give the $(G\ltimes_\tau\bb{C})$-Hilbert module structure on $G\ltimes_{-\tau}\bb{C}$ for $f$, $f_1$, $f_2\in G\ltimes_{-\tau}\bb{C}$ and $b\in G\ltimes_\tau\bb{C}$.
Clearly $\Dirac_L$ corresponds to $\Dirac_R$ under the isomorphism $\Phi$.
\end{proof}

We add a small remark about the above. We can rewrite the isomorphism $\Phi$ in the operator language. One can check the following by a simple calculation.
\begin{lem}\label{small remark on a-ind}
Let $(V,\sigma)$ be a $\tau$-twisted representation of $G$ at level $1$. Then the dual space $(V^*,\sigma^*)$ is at level $-1$ by the dual representation $\sigma^*(g)={}^t\sigma(g^{-1})$. We can define standard homomorphisms
$$\Op^+:G\ltimes_\tau \bb{C}\to \End(V^*),\;\;\Op^-:G\ltimes_{-\tau} \bb{C}\to \End(V).$$
If we define a linear map $\Psi:\End(V)\to\End(V^*)$ by the transpose of operators, the following diagram commutes:
$$\begin{CD}
G\ltimes_{-\tau}\bb{C} @>\Op^->> \End(V) \\
@V\Phi VV @VV\Psi V \\
G\ltimes_{\tau}\bb{C} @>\Op^+>> \End(V^*).
\end{CD}$$
Thanks to the diagram, we can induce an $\End(V^*)$-Hilbert module structure on $\End(V)$ as follows:
$$f\cdot b:={}^tb\circ f$$
$$\innpro{f_1}{f_2}{}={}^t(f_2\circ f_1^*)$$
for $f$, $f_1$, $f_2\in \End(V)$, $b\in \End(V^*)$.

More concretely, $\Psi(v\otimes f)$ is nothing but $f\otimes v$ for $v\in V=V^{**}$ and $f\in V^*$.

\end{lem}

For infinite-dimensional cases, the function model for a group $C^*$-algebra does not work, but the operator model does work. So the above plays an important role to study the analytic index for infinite-dimensional cases.

\section{Infinite-dimensional case}

\subsection{The setting, the factorization and the main result}

We work under Assumption \ref{setting}. Firstly, we give the $U(1)$-central extension $\tau$ which we study.

\begin{dfn}[\cite{FHTII}, \cite{T}]\label{central extension}
Under the canonical decomposition $LT=T\times \Omega T=T\times\Pi_T\times \Omega T_0$, where $\Omega T$ is the group consisting of base point-preserving loops, $\Pi_T$ is the set of connected components of $\Omega T$, that is, the set of rotation numbers, and $\Omega T_0$ is the identity component of $\Omega T$. For a fixed $k\in\bb{Z}-\{0\}$, $\tau:LT\times LT\to U(1)$ is the $2$-cocycle defined by the formula
\begin{itemize}
\item $\tau(T\times\Pi_T,\Omega T_0)=\tau(\Omega T_0,T\times \Pi_T)=1$,
\item $\tau((t_1,n_1),(t_2,n_2))=t_2^{kn_1}$, and 
\item $\tau(l_1,l_2)=\exp\bra{i\int l_1\frac{dl_2}{d\theta}d\theta}$.
\end{itemize}
The $U(1)$-central extension of $LT$ is given by $LT\times U(1)$ as a space, and
$$(l_1,z_1)\cdot (l_2,z_2):=(l_1l_2,z_1z_2\tau(l_1,l_2))$$
as a group. $LT^\tau$ denotes the $U(1)$-central extension of $LT$ given above.

\end{dfn}
\begin{lem}[\cite{FHTII}]
We can restrict the $U(1)$-central extension of $LT$ to $\Omega T_0$ and $T\times  \Pi_T$. Using the same symbol for the restrictions of $\tau$,
$$LT^\tau\cong (T\times\Pi_T)^\tau\boxtimes \Omega T_0^\tau.$$
\end{lem}

We study an infinite-dimensional manifold $\ca{M}$ equipped with an $LT$-action, and a proper equivariant map $\Phi:\ca{M}\to L\fra{t}^*$, where $LT$ acts on $L\fra{t}^*$ via the gauge action: $l.A:=A+l^{-1}dl$, as mentioned in Assumption \ref{setting}.
%
%
%
The commutativity of $T$ implies a simplification as in \cite{T}. As a result, we get the following factorization.

\begin{pro}[\cite{T}]
We can define a compact manifold $M:=\ca{M}/\Omega T$ and a $\Pi_T$-principal bundle $\widetilde{M}\to M=\ca{M}/\Omega T_0$, that is, the following diagram commutes:
$$\begin{CD}
\ca{M} @>>> \widetilde{M} @>>> M \\
@V\Phi VV @V\widetilde{\phi}VV @VV\phi V \\
L\fra{t}^* @>>\int> \fra{t} @>>\exp> T.
\end{CD}$$
Since $T$ commutes with $\Omega T$, $T$ acts on $M$ and $\phi:M\to T$ is $T$-invariant.
Moreover, since the bundle $\int:L\fra{t}^* \to \fra{t}$ is trivial, we can trivialize $\ca{M} \to \widetilde{M}$ canonically: $\ca{M}\cong\widetilde{M}\times \Omega T_0$.

Briefly, $\ca{M}$ is the product of $\widetilde{M}$ and $\Omega T_0$ including the group action and the twisting.
\end{pro}






Thanks to the decomposition, we can concentrate at $\Omega T_0$-part. 
To prove the main result Theorem \ref{Main Result}, the following is sufficient.
Let $x$ be a pair $(L^2(\Omega T_0,\tau)\otimes S_{\Omega T_0},\cancel{\partial}_R)$. We regard $x$ as a ``virtual $K$-homology class''. It is not a real $K$-homology class, because we have not defined the $C^*$-algebra $C_0(\ca{M})$ so far. But we can define several real objects related to $x$ as follows. We state it without detailed definitions.

\begin{thm}
We can define:
\begin{itemize}
\item a $C^*$-algebra which can be regarded as $\Omega T_0\ltimes C_0(\Omega T_0)$,
\item the ``image of $x$ along the partial $j$-homomorphism'' $j_\tau^{\Omega T_0}(x)\in \Psi(\Omega T_0\ltimes C_0(\Omega T_0),\Omega T_0\ltimes_\tau \bb{C})$, and
\item the ``Mishchenko line bundle'' $[c]\in \Psi(\bb{C},\Omega T_0\ltimes C_0(\Omega T_0))$.
\end{itemize}
As a result, we can define a $KK$-theoretical index $\mu^{\Omega T_0}_\tau(x)\in \Psi(\bb{C},\Omega T_0\ltimes_\tau \bb{C})$. It is represented by the completion of $(S_{\Omega T_0}\otimes (\Omega T_0\ltimes_\tau\bb{C}),\cancel{\partial}_L)$. 

The analytic index is represented by the completion of $((\Omega T_0\ltimes_{-\tau}\bb{C})\otimes S_{\Omega T_0},\Dirac_R)$, and two indices are isomorphic to one another in the sense of $KK$.


\end{thm}
\begin{rmk}
In fact, $L^2(\Omega T_0,\tau)$, $S_{\Omega T_0}$ and $\Dirac_L$ have been constructed in \cite{T}.
\end{rmk}

\subsection{The virtual $K$-homology class $(L^2(\Omega T_0,\tau)\otimes S_{\Omega T_0},\cancel{\partial}_R)\in $``$KK_{\Omega T_0}^\tau(C_0(\Omega T_0),\bb{C})$''}

We would like to define $j^{\Omega T_0}_\tau(L^2(\Omega T_0,\tau)\otimes S_{\Omega T_0},\Dirac_R)\in KK(\Omega T_0\ltimes C_0(\Omega T_0),\Omega T_0\ltimes_\tau \bb{C})$ by an analogy of Section 3.2,
$$(L^2(\Omega T_0)\otimes S_{\Omega T_0}\otimes (\Omega T_0\ltimes_\tau \bb{C}),D\otimes_2\id+\id\otimes_1\Dirac_L),$$
where $F_1\otimes_1 F_2$ is defined by $F_1\otimes_1 F_2(v\otimes s\otimes f)=F_1(v)\otimes F_2(s\otimes f)$ for $v\in L^2(\Omega T_0)$, $s\in S_{\Omega T_0}$, $f\in L^2(\Omega T_0,\tau)$, $F_1:L^2(\Omega T_0)\to L^2(\Omega T_0)$ and $F_2: S_{\Omega T_0}\otimes (\Omega T_0\ltimes_\tau \bb{C})\to S_{\Omega T_0}\otimes (\Omega T_0\ltimes_\tau \bb{C})$, and $F_3\otimes_2 F_4$ is defined similarly.
We recall the construction of $L^2(\Omega T_0,\tau)$, $S_{\Omega T_0}$, $\Dirac_L$ and $\Omega T_0\ltimes_{\pm\tau} \bb{C}$ in \cite{T} in this subsection, and we will define $L^2(\Omega T_0)$, $\Omega T_0\ltimes C_0(\Omega T_0)$ and $D$ in the next subsection.


\begin{dfn}
Since $\Omega T_0$ can be identified with $\{f\in C^\infty(S^1,\fra{t})\mid \int f(\theta)d\theta=0\}$, we can choose a C.O.N.S. of the Lie algebra $\Lie(\Omega T_0)$ by
$$\frac{\cos(n\theta)}{\sqrt{n\pi}},\frac{\sin(n\theta)}{\sqrt{n\pi}}$$
with respect to the inner product $\innpro{X}{Y}{}=\omega(X,J(Y))$, where $\omega(X,Y)=\int X\frac{dY}{d\theta}d\theta$, $J(\cos)=\sin$, and $J(\sin)=-\cos$. The rigid rotation on $\Omega T_0$ defines an operator $d:\Lie(\Omega T_0)\to \Lie(\Omega T_0)$.

We can define a complex basis on $\Lie(\Omega T_0)\otimes \bb{C}$ by
$$z_n:=\frac{1}{\sqrt{2}}\bra{\frac{\cos(n\theta)}{\sqrt{n\pi}}+i\frac{\sin(n\theta)}{\sqrt{n\pi}}}, \overline{z_n}:=\frac{1}{\sqrt{2}}\bra{\frac{\cos(n\theta)}{\sqrt{n\pi}}-i\frac{\sin(n\theta)}{\sqrt{n\pi}}}.$$
Notice that $dz_n=inz_n$ and $d\overline{z_n}=-in\overline{z_n}$. A finite linear combination of $z_n$'s and $\overline{z_n}$'s is called a {\bf finite energy vector}, and $\bra{\Lie(\Omega T_0)\otimes \bb{C}}_\fin$ denotes the set of finite energy vectors. Clearly it is a dense subspace of $\Lie(\Omega T_0)\otimes \bb{C}$.

\end{dfn}

\begin{dfn}[\cite{PS}]
Let $L^2(\bb{R}^\infty)_\fin$ be the algebraic symmetric algebra of $\bigoplus_{n>0} \bb{C}z_n$. It admits an inner product 
$$\innpro{z_1^{k_1}z_2^{k_2}\cdots}{z_1^{l_1}z_2^{l_2}\cdots}{}:=k_1!k_2!\cdots \delta_{k_1,l_1}\delta_{k_2,l_2}\cdots.$$
$\fra{u}(L^2(\bb{R}^\infty)_\fin)$ denotes the set of skew symmetric operators on $L^2(\bb{R}^\infty)_\fin$.
We can define a linear map $d\rho:\Lie(\Omega T_0)\to \fra{u}(L^2(\bb{R}^\infty))$ by the restriction of the following formula:
\begin{align*}
d\rho(z_n)(z_1^{k_1}z_2^{k_2}\cdots z_n^{k_n}\cdots)&:=z_1^{k_1}z_2^{k_2}\cdots z_n^{k_n+1}\cdots, \\
d\rho(\overline{z_n})(z_1^{k_1}z_2^{k_2}\cdots z_n^{k_n}\cdots)&:=-k_nz_1^{k_1}z_2^{k_2}\cdots z_n^{k_n-1}\cdots.
\end{align*}
These operators satisfy the commutation relation to be an infinitesimal $\tau$-twisted representation: $[d\rho(z_n),d\rho(\overline{z_n})]=\id$.
Moreover, $L^2(\bb{R}^\infty)_\fin$ admits an operator $d\rho(d)$ defined by $d\rho(d)(z_1^{k_1}z_2^{k_2}\cdots)=(i\sum jk_j)z_1^{k_1}z_2^{k_2}\cdots$. 
The following positive energy condition is satisfied:
\begin{itemize}
\item $[d\rho(d),d\rho(z_n)]=ind\rho(z_n)$, $[d\rho(d),d\rho(\overline{z_n})]=-ind\rho(\overline{z_n})$, and 
\item $\frac{1}{i}d\rho(d)$ has discrete spectrum and positive definite.
\end{itemize}

Taking the completion with respect to the above inner product, we get a Hilbert space $L^2(\bb{R}^\infty)$.
In fact, these operators c]ome from the unique irreducible positive energy representation of $\Omega T_0$.
The energy operator $d\rho(d)$ can be written as $-i\sum nd\rho(z_n)d\rho(\overline{z_n})$.
\end{dfn}

We adopt the following easy definition. In fact, one can understand an element of the following algebra as an ``asymptotically Gaussian function'' as in \cite{T}.

\begin{dfn}
Let $\Omega T_0\ltimes_\tau \bb{C}$ be the $C^*$-algebra $\ca{K}(L^2(\bb{R}^\infty)^*)$, 
and let $\Omega T_0\ltimes_{-\tau}\bb{C}$ be $\ca{K}(L^2(\bb{R}^\infty))$.
We sometimes regard $\Omega T_0\ltimes_{\tau}\bb{C}$ 
($\Omega T_0\ltimes_{-\tau}\bb{C}$) as a certain completion of $L^2(\bb{R}^\infty)_\fin^*\otimes^\alg L^2(\bb{R}^\infty)_\fin$ ($L^2(\bb{R}^\infty)_\fin\otimes^\alg L^2(\bb{R}^\infty)_\fin^*$ respectively).

\end{dfn}

We can also define the Spinor space for $\Lie(\Omega T_0)$, thanks to the complex structure coming from the rigid rotation.

\begin{dfn}[\cite{FHTII}]
Let $S_{\Omega T_0,\fin}$ be the algebraic exterior algebra of $\bigoplus_{n>0} \bb{C}\overline{z_n}$. It has an inner product defined by $\innpro{\overline{z_{k_1}}\wedge \overline{z_{k_2}}\wedge\cdots \wedge \overline{z_{k_p}}}{\overline{z_{l_1}}\wedge \overline{z_{l_2}}\wedge\cdots \wedge \overline{z_{l_p}}}{}:=\delta_{k_1,l_1}\delta_{k_2,l_2}\cdots\delta_{k_p,l_p}$.
The space $S_{\Omega T_0,\fin}$ also admits a Clifford multiplication of $\Lie(\Omega T_0)\otimes \bb{C}$ by
\begin{align*}
\gamma(\overline{z_n}) &:= \sqrt{2}\overline{z_n}\wedge, \\
\gamma(z_n) &:= -\sqrt{2}\overline{z_n}\rfloor,
\end{align*}
where $\rfloor$ denotes the interior product. Let $S_{\Omega T_0}$ be the completion of $S_{\Omega T_0,\fin}$ with respect to the above inner product.

The exterior algebra has a specific vector $1$. To distinguish with the lowest weight vector belonging to the symmetric algebra, we introduce the symbol $1_f$. ``$f$'' comes from ``fermion''.

Let $N$ be the number operator defined by $N(\overline{z_{k_1}}\wedge \overline{z_{k_2}}\wedge\cdots \wedge \overline{z_{k_p}}):=(\sum k_j)\overline{z_{k_1}}\wedge \overline{z_{k_2}}\wedge\cdots \wedge \overline{z_{k_p}}$. It satisfies the equality $N=-\frac{1}{2}\sum n\gamma(\overline{z_n})\gamma(z_n)$. 

\end{dfn}

\begin{dfn}
Let $L^2(\bb{R}^\infty)^*$ be the dual space of $L^2(\bb{R}^\infty)$. 
We regard $L^2(\bb{R}^\infty)^*$ as a completion of the symmetric algebra of $\bigoplus_{n>0} \bb{C}\overline{z_n}$.
It has the dual action of $\Lie(\Omega T_0)\otimes \bb{C}$, and $d\rho^*(d)$. Let $L^2(\bb{R}^\infty)_\fin^*$ be the set of finite sums of $d\rho^*(d)$-eigen vectors, which is a dense subspace.
Let $L^2(\Omega T_0,\tau)$ be the tensor product $L^2(\bb{R}^\infty)\otimes L^2(\bb{R}^\infty)^*$ as a Hilbert space, and let $L^2(\Omega T_0,\tau)_\fin$ be the algebraic tensor product $L^2(\bb{R}^\infty)_\fin\otimes^\alg L^2(\bb{R}^\infty)_\fin^*$.

The ``Dirac operator on $\Omega T_0$'' acting on $L^2(\Omega T_0,\tau)_\fin\otimes^\alg S_{\Omega T_0,\fin}$ is defined by
$$\Dirac_R:=\sum\sqrt{n}\Bigl(\id\otimes d\rho^*(\overline{z_n})\otimes 
\gamma(z_n)+
\id\otimes d\rho^*(z_n)\otimes \gamma(\overline{z_n})\Bigr).$$

This operator preserves the subspace $L^2(\Omega T_0,\tau)_\fin\otimes^\alg S_{\Omega T_0,\fin}$.
\end{dfn}

\begin{rmk}
It is not clear that $\Dirac_R$ makes sense, because is is defined by an infinite sum of operators. However, for any elements of $v\in S_{\Omega T_0,\fin}\otimes L^2(\Omega T_0,\tau)_\fin$, $\Dirac_R(v)$ is a finite sum, hence $\Dirac_R$ can be defined. Our new constructions $L^2(\Omega T_0)$ and $D$ require a quantitative version of this argument.
\end{rmk}

There is a Weitzenb\"ock type formula. One can verify it by a complicated computation (\cite{T}).


\begin{lem}
$$\Dirac_R^2=2\bra{\id\otimes \id\otimes N+\id\otimes \frac{d\rho^*(d)}{i}\otimes \id}.$$

in particular it is positive definite.

\end{lem}

We call the eigenvalue of $\Dirac_R^2$ the {\bf energy}.
Let us prepare an estimate which we need later.
 
\begin{lem}\label{estimate about PER}
When $\phi=v\otimes \overline{z_1}^{k_1} \overline{z_2}^{k_2}\cdots \overline{z_p}^{k_p}\otimes s\in L^2(\bb{R}^\infty)\otimes^\alg L^2(\bb{R}^\infty)^*\otimes^\alg S_{\Omega T_0,\fin}$ is at energy $\lambda^2$, then $\|d\rho^*(z_n)\phi\|\leq \frac{|\lambda|}{\sqrt{2n}}\|\phi\|$, and
$\|d\rho^*(\overline{z_n})\phi\|\leq \bra{\frac{|\lambda|}{\sqrt{2n}}+1}\|\phi\|$.
\end{lem}
\begin{proof}
\begin{align*}
\innpro{d\rho^*(z_n)\phi}{d\rho^*(z_n)\phi}{}&=\innpro{(d\rho^*(z_n))^*d\rho^*(z_n)\phi}{\phi}{}\\
&=\innpro{v}{v}{}\innpro{(d\rho^*(z_n))^*d\rho^*(z_n)\overline{z_1}^{k_1} \overline{z_2}^{k_2}\cdots \overline{z_p}^{k_p}}{\overline{z_1}^{k_1} \overline{z_2}^{k_2}\cdots \overline{z_p}^{k_p}}{}\innpro{s}{s}{} \\
&= k_n\|\phi\|^2.
\end{align*}
Since $2\sum mk_m\leq \lambda^2$, we get $k_n\leq \frac{\lambda^2}{2n}$. Therefore,
$$\|d\rho^*(z_n)\phi\|=\sqrt{k_n}\|\phi\|\leq \frac{|\lambda|}{\sqrt{2n}}\|\phi\|.$$
\end{proof}

Let us recall that the isomorphism in Proposition \ref{j-hom for locally compact} involves a correspondence $\phi\mapsto \phi^{\lor}$, where $\phi^\lor(g):=\phi(g^{-1})$. This is why $\Dirac_R$ changes to $\Dirac_L$.

Let us introduce $\Dirac_L$ here.
\begin{dfn}
Let us consider a $(\Omega T_0\ltimes_\tau \bb{C})$-Hilbert module $S_{\Omega T_0}\otimes (\Omega T_0\ltimes_\tau \bb{C})$ and its subspace $S_{\Omega T_0,\fin}\otimes^\alg L^2(\bb{R}^\infty)_\fin^*\otimes^\alg L^2(\bb{R}^\infty)$.
Let 
$$\Dirac_L:=\sum_n \sqrt{n}\Bigl(\gamma(z_n)\otimes d\rho^*(\overline{z_n})\otimes \id+\gamma(\overline{z_n})\otimes d\rho^*(z_n)\otimes \id\Bigr)$$
be an operator acting on $S_{\Omega T_0,\fin}\otimes^\alg L^2(\bb{R}^\infty)_\fin^*\otimes^\alg L^2(\bb{R}^\infty)$. We use the same symbol for the closure of $\Dirac_L$.

\end{dfn}

\subsection{The Hilbert space ``$L^2(\Omega T_0)$'' and the ``Dirac operator'' $D$}



Let us notice that $L^2(\Omega T_0,\tau)\otimes S_{\Omega T_0}$ is the ``infinite tensor product'' of $L^2(\bb{R}^2,\tau)\otimes S_{\bb{R}^2}$'s, and the Dirac operator is also the infinite product of $\Dirac_{\bb{R}^2}$'s. This construction works, because $\Dirac_{\bb{R}^2}$ has the one dimensional kernel, and $L^2(\Omega T_0,\tau)\otimes S_{\Omega T_0}$ can be defined by the inductive limit of 
$$\cdots\to L^2(\bb{R}^{2n},\tau)\otimes S_{\bb{R}^{2n}} \cong \bra{L^2(\bb{R}^{2n},\tau)\otimes S_{\bb{R}^{2n}}}\otimes \ker(\Dirac_{\bb{R}^2})\hookrightarrow  L^2(\bb{R}^{2n+2},\tau)\otimes S_{\bb{R}^{2n+2}}\to\cdots.$$
This is the reason why $\Dirac_R$ can be defined.

However, the standard Dirac operator $D$ acting on $L^2(\bb{R}^{2n})\otimes S_{\bb{R}^{2n}}$ has no nonzero kernel. Therefore we can not naively consider the infinite tensor product of $D_{\bb{R}^2}$ acting on $L^2(\bb{R}^{2})\otimes S_{\bb{R}^{2}}$. 
So we need to to choose an appropriate construction of $L^2(\Omega T_0)$.

\begin{dfn}
Let $\chi_{\sigma}$ be the function on $\bb{R}^2$ defined by
$$
\chi_{\sigma}(x)=\begin{cases}
\frac{1}{\sqrt{\pi\sigma^2}} & |x|\leq \sigma \\
0 & |x|> \sigma,
\end{cases}
$$
for $\sigma>0$, and let $\Xi_{\sigma}$ be the Fourier transform $\ca{F}(\chi_{\sigma})$. Since $\chi_\sigma$ is a unit vector, so is $\Xi_{\sigma}$.
We can construct a sequence of isometric embeddings
$$ L^2(\bb{R}^2)\hookrightarrow L^2(\bb{R}^4)\hookrightarrow\cdots$$
by $i_k:L^2(\bb{R}^{2k})\ni f\mapsto f\otimes \Xi_{\sigma_{k+1}}\in L^2(\bb{R}^{2k+2})$.
Let $\{\sigma_k\}$ be a sequence of positive numbers satisfying that
\begin{equation}\label{eq:behavior of sigma}
\sum_k\sqrt{k}\sigma_k<\infty.
\end{equation}
It is a key condition.
For example, $\sigma_k=2^{-k}$ satisfies the condition (\ref{eq:behavior of sigma}).
We define $L^2(\Omega T_0)$ by the inductive limit: $L^2(\Omega T_0):=\varinjlim L^2(\bb{R}^{2k})$.

Let $L^2_\infty(\bb{R}^{2k}):=\bbra{\ca{F}(\phi)\in L^2(\bb{R}^{2k})\mid 
p\phi\in L^2(\bb{R}^{2k})\text{ for any polynomial }p}$, and $L^2_\infty(\Omega T_0)$ be the algebraic inductive limit $\varinjlim^\alg L^2_\infty(\bb{R}^{2k})$. An element of this space can be regarded as an ``{\bf asymptotically constant function}''.

We sometimes use a slightly loose symbol $\Xi=\Xi_{\sigma_{n+1}}\otimes \Xi_{\sigma_{n+2}}\otimes \cdots.$

\end{dfn}

Thanks to the good choice of $\sigma_k$'s, we can define a Dirac operator $D$ not as a formal infinite sum, but as an operator.
Let $dR_{z_n}:L^2_\infty(\Omega T_0)\to L^2_\infty(\Omega T_0)$ be the operator defined by 
$$\frac{1}{\sqrt{2}}\bra{\frac{\partial}{\partial x_n}+i\frac{\partial}{\partial y_n}}$$
and let $dR_{\overline{z_k}}$ be its complex conjugate, where $x_n=\frac{\cos(n\theta)}{\sqrt{n\pi}}$ and $y_n={\frac{\sin(n\theta)}{\sqrt{k\pi}}}$. Strictly speaking, if $j_n:L^2_\infty(\bb{R}^{2m})\hookrightarrow L^2_\infty(\Omega T_0)$ is the embedding defining the limit, $dR_{z_n}\bra{j_m(f)}$ is defined by $j_m(dR_{z_n}f)$ for sufficiently large $m$. 

Let us introduce the Dirac operator $D$ acting on $L^2(\Omega T_0)\otimes S_{\Omega T_0}$.

\begin{dfn}
The Dirac operator is defined by
$$D:=\sum_{n=1}^{n=\infty} \sqrt{n}\Bigl( dR_{z_n}\otimes\gamma(\overline{z_n})+dR_{\overline{z_n}}\otimes\gamma(z_n)\Bigr)$$
and its approximation is defined by
$$D_N^M :=\sum_{n=N}^{n=M} \sqrt{n}\Bigl( dR_{z_n}\otimes\gamma(\overline{z_n})+dR_{\overline{z_n}}\otimes\gamma(z_n)\Bigr)$$
for $N<M\leq \infty$.

Considering the isomorphism $L^2(\Omega T_0)\cong L^2(\bb{R}^{2N})\otimes L^2(\bb{R}^{2N\perp})$ and $S_{\Omega T_0}\cong S_{\bb{R}^{2N}}\otimes S_{\bb{R}^{2N\perp}}$, we can regard $D$ as $D_1^N\otimes \id+\id\otimes D_{N+1}^\infty$.
\end{dfn}
\begin{pro}\label{convergence}
The operator $D$ does make sense on $L^2_\infty(\Omega T_0)\otimes^\alg S_{\Omega T_0,\fin}$, and it is essentially self-adjoint. Moreover, $D^2$ also makes sense on $L^2_\infty(\Omega T_0)\otimes S_{\Omega T_0}$.

\end{pro}
\begin{proof}
Let us make $D$ act on an element $\phi:=\bigl[f\otimes \Xi_{m+1}\otimes \Xi_{m+2}\otimes \cdots\bigr]\otimes s$. 
Since $R_{x_n}$ corresponds to the multiplication by $ix_n$ by the Fourier transform, 
$\|dR_{z_n}\Xi_{\sigma_n}\|\leq \sigma_n\|\Xi_{\sigma_n}\|$. Therefore, by the triangle inequality, and $\|\gamma(z_n)\|_\op\leq\sqrt{2}$,
\begin{align*}
&\sum_{n=M+1}^{\infty} \sqrt{n}\|dR_{z_n}\otimes\gamma(\overline{z_n})\bbbra{\phi}+dR_{\overline{z_n}}\otimes\gamma(z_n)\bbbra{\phi}\|\\
&\;\;\;\;\leq \sum_{n=M+1}^\infty 2\sqrt{2n}\sigma_n<\infty,
\end{align*}
%
thanks to the key condition (\ref{eq:behavior of sigma}). In particular, the infinite sum $D\phi$ converges.

Notice that $D$ is a symmetric operator.
To be self-adjoint, it suffices to check that $D\pm i$ has dense range. Let a non-zero element $f\otimes s\in L^2(\Omega T_0)_\fin\otimes^\alg S_{\Omega T_0\fin}$ be given.
For any positive $\varepsilon$, we would like to find an approximate solution $\phi_\pm\in L^2(\Omega T_0)_\fin\otimes^\alg S_{\Omega T_0\fin}$ satisfying that $\|(D\pm i)\phi_\pm-f\otimes s)\|<2\varepsilon$. 
We may assume that $\varepsilon <\frac{\|f\otimes s\|}{2}$.
To find $\phi_\pm$, find $M\in\bb{N}$ such that $\sum_{n>M}2\sqrt{2n}\sigma_n<\frac{2\varepsilon}{\|f\otimes s\|}$. It is always possible from the Assumption \ref{eq:behavior of sigma}. 
Retaking $M$ greater if necessary, we can assume that $f$ is of the form $f_0\otimes \Xi_{\sigma_{M+1}}\otimes \Xi_{\sigma_{M+2}}\otimes \cdots$ for $f_0\in L^2_\infty(\bb{R}^{2M})$, and $s$ is of the form $\overline{z_{k_1}}\wedge \overline{z_{k_2}}\wedge \cdots \wedge \overline{z_{k_p}}$ for $p\leq M$.
Since $D_1^{M}$ is self-adjoint on $L^2(\bb{R}^{2M})\otimes S_{\bb{R}^{2M}}$, there exists $\phi_{0,\pm}\in L^2(\bb{R}^{2M})_\fin\otimes S_{\bb{R}^{2M}}$ satisfying that $\|(D_1^M\pm i)\phi_{0,\pm}-f_0\otimes s\|<\varepsilon$. 
Since $D_1^M\pm i$ does not reduce the norm, $\|\phi_{0,\pm}\|<\|f_0\otimes s\|-\epsilon<\frac{\|f\otimes s\|}{2}$.


Combining these estimates, we get the following:
\begin{align*}
&\|(D\pm i)\bra{\phi_{0,\pm}\otimes (\Xi\otimes 1_f)}-(f_0\otimes \Xi)\otimes s\| \\
&< \|[(D_1^M\pm i)\phi_{0,\pm}-f_0\otimes s]\otimes (\Xi\otimes 1_f)\|+\|(-1)^{\deg(\phi_{0,\pm})}\phi_{0,\pm}\otimes D_{M+1}^\infty(\Xi\otimes 1_f)\| \\
&< \|(D_1^M\pm i)\phi_{0,\pm}-f_0\otimes s\|\|\Xi\otimes 1_f\|+\|\phi_{0,\pm}\|\|D_{M+1}^\infty(\Xi\otimes 1_f)\| \\
&<2\varepsilon.
\end{align*}



Let us estimate $\|\sqrt{nm}dR_{z_n}dR_{z_m}f\otimes \gamma(\overline{z_n})\gamma(\overline{z_m})s\|$'s. The rest part (e.g. $\|\sqrt{nm}dR_{z_n}dR_{\overline{z_m}}f\otimes \gamma(\overline{z_n})\gamma(z_m)s\|$'s) can be estimated by the same method. Since $\gamma(z_n)$'s are bounded operators, and the operator norm is $\sqrt{2}$, we may ignore $S$-part.
So we estimate $\|dR_{z_n}dR_{z_m}f\|$.

By definition of $L^2_\infty(\Omega T_0)$, $f$ can be written as $f_0\otimes \Xi$ for $f_0\in C_c^\infty(\bb{R}^{2M})$. Let $\|\cdot\|_{L^2_k}$ be the $k$-th Sobolev norm on $\bb{R}^{2M}$.
If $n$, $m>M$, $\|\sqrt{nm}dR_{z_n}dR_{z_m}f\|<\sqrt{nm}\sigma_n\sigma_m\|f_0\|_{L^2}$. 
If $n>M$ and $m\leq M$, $\|\sqrt{nm}dR_{z_n}dR_{z_m}f\|<\sqrt{nm}\sigma_n\|f_0\|_{L^2_1}$. If $n,m\leq M$, $\|\sqrt{nm}dR_{z_n}dR_{z_m}f\|\leq \sqrt{nm}\|f_0\|_{L^2_2}$. 
Since
\begin{align*}
&\sum_{n,m} \|\sqrt{nm}dR_{z_n}dR_{z_m}f\otimes \gamma(\overline{z_n})\gamma(\overline{z_m})s\|\\
&\leq 2\bra{
\sum_{n,m>M}\sqrt{nm}\sigma_n\sigma_m\|f_0\|_{L^2}
+2\sum_{n>M,m\leq M}\sqrt{nm}\sigma_n\|f_0\|_{L^2_1}
+\sum_{n,m\leq M}\sqrt{nm}\|f_0\|_{L^2_2}
},
\end{align*}
the infinite sum converges, thanks to Assumption \ref{eq:behavior of sigma}. Hence $D^2(f\otimes s)$ makes sense for $f\otimes s\in L^2_\infty(\Omega T_0)\otimes^\alg S_{\Omega T_0,\fin}$.


\end{proof}

Using the definition of $L^2(\Omega T_0)$, we define the $C^*$-algebra $\Omega T_0\ltimes C_0(\Omega T_0)$, by a similar technique in \cite{T}. For the details of classical crossed product algebras, see \cite{Wil} for example.

\begin{dfn}
Noticing that $\bb{R}^{2n}\ltimes C_0(\bb{R}^{2n})\cong \ca{K}\bra{L^2\bra{\bb{R}^{2n}}}$, we define the $C^*$-algebra $\Omega T_0\ltimes C_0(\Omega T_0)$ by the inductive limit of $\cdots \to \ca{K}\bra{L^2\bra{\bb{R}^{2n}}}\to \ca{K}\bra{L^2\bra{\bb{R}^{2n+2}}}\to\cdots$ by the $*$-homomorphism 
$$k\mapsto k\otimes P_{\Xi_{\sigma_{n+1}}},$$
where $P_{\Xi_{\sigma_{n+1}}}$ is the rank one projection onto $\bb{C}\Xi_{\sigma_{n+1}}$.

$\ca{K}\bra{L^2\bra{\bb{R}^{2n}}}$ contains a dense $*$-subalgebra $\ca{K}\bra{L^2\bra{\bb{R}^{2n}}}_\fin:=L^2_\infty(\bb{R}^{2n})\otimes^\alg L^2_\infty(\bb{R}^{2n})^*$.
The above $*$-homomorphism defines a $*$-homomorphism from $\ca{K}\bra{L^2\bra{\bb{R}^{2n}}}_\fin$ to $\ca{K}\bra{L^2\bra{\bb{R}^{2n+2}}}_\fin$, hence we can define a dense $*$-subalgebra $\bra{\Omega T_0\ltimes C_0(\Omega T_0)}_\fin$ by the algebraic inductive limit.

\end{dfn}

\subsection{The construction of $j^{\Omega T_0}_\tau(x)\in \Psi(\Omega T_0\ltimes C_0(\Omega T_0),\Omega T_0\ltimes_\tau \bb{C})$}
Combining the above two subsections, we will construct the main objects, and prove the main theorem.

\begin{thm}\label{j-hom for Omega T0}
The closure of the operator 
$\widetilde{\Dirac}_R:=D\otimes_2\id+\id\otimes_1\cancel{\partial}_L$ acting on $L^2(\Omega T_0)_\fin\otimes^\alg S_{\Omega T_0,\fin}\otimes^\alg L^2(\Omega T_0,\tau)_\fin$
is self-adjoint and regular, and
$$(L^2(\Omega T_0)\otimes S_{\Omega T_0}\otimes (\Omega T_0\ltimes_\tau \bb{C}),\widetilde{\Dirac}_R)\in\Psi(\Omega T_0\ltimes C_0(\Omega T_0),\Omega T_0\ltimes_\tau \bb{C}),$$
that is, $[\widetilde{\Dirac}_R,a]\in \ca{L}_{\Omega T_0\ltimes_\tau \bb{C}}(L^2(\Omega T_0)\otimes S_{\Omega T_0}\otimes (\Omega T_0\ltimes_\tau \bb{C}))$ for $a\in (\Omega T_0\ltimes C_0(\Omega T_0))_\fin$, and $(1+\widetilde{\Dirac}_R^2)^{-1}a\in\ca{K}_{\Omega T_0\ltimes_\tau \bb{C}}(L^2(\Omega T_0)\otimes S_{\Omega T_0}\otimes (\Omega T_0\ltimes_\tau \bb{C}))$ for all $a\in \Omega T_0\ltimes C_0(\Omega T_0)$.
\end{thm}
In this subsection, we prove the above theorem, dividing three steps: $(a)$ $\widetilde{\Dirac}_R$ is self-adjoint and regular, $(b)$ $[\widetilde{\cancel{\partial}_R},a\otimes_1 \id]$ is bounded for all $a\in (\Omega T_0\ltimes C_0(\Omega T_0))_\fin$, and $(c)$ $(1+\widetilde{\cancel{\partial}_R}^2)^{-1}a\otimes_1 \id$ is compact for all $a\in \Omega T_0\ltimes C_0(\Omega T_0)$.

\begin{lem}
$\widetilde{\Dirac}_R$ is self-adjoint and regular.
\end{lem}
\begin{proof}
It is sufficient to check that $\widetilde{\Dirac}_R\pm i$ has dense range. 
We use a similar argument of the proof of Proposition \ref{convergence}. Let us consider the isomorphism
\begin{align*}
&L^2(\Omega T_0)\otimes S_{\Omega T_0}\otimes (\Omega T_0\ltimes_\tau \bb{C}) \\
&\cong 
L^2(\bb{R}^{2M})\otimes S_{\bb{R}^{2M}}\otimes (\bb{R}^{2M}\ltimes_\tau \bb{C})
\bigotimes L^2(\bb{R}^{2M\perp})\otimes S_{\bb{R}^{2M\perp}}\otimes (\bb{R}^{2M\perp}\ltimes_\tau \bb{C}).
\end{align*}
Under this isomorphism, we can divide $\widetilde{\Dirac}_R$ into two parts $\widetilde{\Dirac}_{R1}\otimes\id+\id\otimes \widetilde{\Dirac}_{R2}$. One can prove the statement if he notices that the first part is self-adjoint, since it can be written as $\widetilde{D}$ for some $D$ for a finite-dimensional manifold $\bb{R}^{2M}$, and the second part is ``small'' by a similar argument of Proposition \ref{convergence}.

\end{proof}
\begin{lem}\label{first order}
$[\widetilde{\cancel{\partial}_R},a\otimes_1 \id]$ is bounded for all $a\in (\Omega T_0\ltimes C_0(\Omega T_0))_\fin$.
\end{lem}
\begin{proof}

Since $a\otimes_1 \id$ commutes with $\id\otimes_1\Dirac_R$, it is sufficient to deal with $[a\otimes \id,D]$. Hence we can focus on $L^2(\Omega T_0)\otimes S_{\Omega T_0}$-part, and we can study in the Hilbert space language, not in the Hilbert module's.
We may assume that $a$ is a Schatten form $\phi\otimes \psi^*$, and $\phi$ ($\psi$) is of the form $\phi_1\otimes \Xi_{\sigma_{M+1}}\otimes \Xi_{\sigma_{M+2}}\otimes\cdots$ ($\psi_1\otimes \Xi_{\sigma_{M+1}}\otimes \Xi_{\sigma_{M+2}}\otimes\cdots$ respectively), where $\phi_1,\psi_1\in L^2(\bb{R}^{2M})$. Then $a$ can be written as $(\phi_1\otimes\psi_1^*) \bigotimes P_{\Xi_{\sigma_{M+1}}\otimes \Xi_{\sigma_{M+2}}\otimes\cdots}$, and we right it as $a_1\otimes P$ simply.

We check that the both of $D\circ a\otimes \id$ and $a\otimes \id \circ D$ are bounded. 
The first boundedness can be verified as follows: the norm of
$$D\circ (a_1\otimes P)\otimes \id(f\otimes s)=D(\phi\otimes s)\innpro{\psi}{f}{},$$
is controlled by $\|f\otimes s\|$.

For the second one, notice Remark \ref{aD is bounded} and the condition \ref{eq:behavior of sigma}.
\end{proof}

For the last step, we recall several properties about Hilbert modules.

\begin{lem}\label{B-cpt op}
As is well known, $\ca{K}_{\Omega T_0\ltimes_\tau \bb{C}}(\Omega T_0\ltimes_\tau \bb{C})\cong \Omega T_0\ltimes_\tau \bb{C}$. Moreover, for $v\otimes w\in L^2(\bb{R}^\infty)^*\otimes L^2(\bb{R}^\infty)$ and $k\in \ca{K}_{\Omega T_0\ltimes_\tau \bb{C}}(\Omega T_0\ltimes_\tau \bb{C})\cong \Omega T_0\ltimes_\tau \bb{C}$, $k(v\otimes w)=k(v)\otimes w$. Roughly speaking, $\Omega T_0\ltimes_\tau \bb{C}$-compact operator on $\Omega T_0\ltimes_\tau \bb{C}$ is of the following form:
\begin{center}
$\bb{C}$-compact operator $\otimes$ $\id$.
\end{center}
\end{lem}
\begin{rmk}
In the equation $k(v\otimes w)=k(v)\otimes w$, $k$ is regarded as an element of $\ca{K}_{\Omega T_0\ltimes_\tau \bb{C}}(\Omega T_0\ltimes_\tau \bb{C})$ in the LHS, and $k$ is regarded as a compact operator acting on $L^2(\bb{R}^\infty)^*$ in the RHS by the isomorphism $\Omega T_0\ltimes_\tau C_0(\Omega T_0)\cong \ca{K}(L^2(\bb{R}^\infty)^*)$.
\end{rmk}

Let us move to the last step.
\begin{lem}
$(1+\widetilde{\cancel{\partial}_R}^2)^{-1}a\otimes_1 \id$ is compact for all $a\in \Omega T_0\ltimes C_0(\Omega T_0)$.

\end{lem}
\begin{proof}
Since the set of compact operators is closed, it is suffices to deal with only general $a$, so we may assume that $a\in (\Omega T_0\ltimes \Omega T_0)_\fin$. Moreover, we may assume that $a$ is of the form $a_1\otimes P$ just as in the proof of Proposition \ref{first order}.

We prepare another operator $\widetilde{\cancel{\partial}}_0:=\id\otimes_1 \Dirac_L$. We will prove the statement by checking that $(1)$ $(1+\widetilde{\cancel{\partial}}_0^2)^{-1}a\otimes_1 \id$ is compact, and $(2)$ the difference $(1+\widetilde{\cancel{\partial}_R}^2)^{-1}a\otimes_1 \id-(1+\widetilde{\cancel{\partial}}_0^2)^{-1}a\otimes_1 \id$ is also compact.

$(1)$ To check the statement, it is enough to focus on the $S_{\Omega T_0}\otimes (\Omega T_0\ltimes_\tau \bb{C})$-part, because $(1+\widetilde{\cancel{\partial}}_0^2)^{-1}a\otimes_1\id=a\otimes_1(1+\Dirac^2_L)^{-1}$ and $a$ acts on $L^2(\Omega T_0)$ as a compact operator.
Recall the description of the $\Omega T_0\ltimes_\tau \bb{C}\cong \ca{K}(L^2(\bb{R}^\infty)^*)$-compact operator on $\Omega T_0\ltimes_\tau \bb{C}$ , under the identification $\ca{K}(L^2(\bb{R}^\infty)^*)\cong \overline{L^2(\bb{R}^\infty)^*\otimes L^2(\bb{R}^\infty)}$, Lemma \ref{B-cpt op}. Then, 
a projection onto $\bb{C}s\otimes \bb{C}\phi\otimes L^2(\bb{R}^\infty)$ from $S_{\Omega T_0}\otimes L^2(\bb{R}^\infty)^*\otimes L^2(\bb{R}^\infty)$ is an $\Omega T_0\ltimes_\tau \bb{C}$-compact operator.
Since $\Dirac_L^2$ acting on the Hilbert space $S_{\Omega T_0}\otimes L^2(\bb{R}^\infty)^*$ has discrete spectrum, we can take a C.O.N.S. consisting of eigenvectors $\{ s_n\otimes \phi_m\}$ such that $2N(s_n)=|s_n|s_n$ and $\frac{2d\rho(d)}{i}\phi_m=|\phi_m|\phi_m$ ($|s_m|$, $|\phi_n|\in\bb{R}$ are nothing but eigenvalues).
By definition of $N$ and $d\rho(d)$, $|s_m|\to \infty$ and $|\phi_n|\to \infty$ as $n,m\to \infty$.
If we wright the projection onto $\bb{C}s\otimes \bb{C}\phi\otimes L^2(\bb{R}^\infty)$ as $P_{s\otimes \phi}$, we can approximate the operator $(1+\Dirac^2_L)^{-1}$ by the sequence of $\Omega T_0\ltimes_\tau \bb{C}$-compact operators
$$\sum_{n,m}^{N,M} \frac{1}{1+|s_n|+|\phi_m|}P_{s_n\otimes \phi_m}.$$
Since two projections $P_{s_n\otimes \phi_m}$ and $P_{s_n'\otimes \phi_m'}$ are orthogonal unless $n=n'$ and $m=m'$, the infinite sum converges in $\ca{K}_{\Omega T_0\ltimes_\tau \bb{C}}(S_{\Omega T_0}\otimes (\Omega T_0\ltimes_\tau \bb{C}))$.



$(2)$ Let us compute the difference $(1+\widetilde{\cancel{\partial}_R}^2)^{-1}a\otimes_1\id- (1+\widetilde{\cancel{\partial}}_0^2)^{-1}a\otimes_1\id$ when $a=(\phi\otimes\psi^*)\otimes P_\Xi$.

\begin{align*}
&\bbbra{(1+\widetilde{\cancel{\partial}_R}^2)^{-1}(\phi\otimes\psi^*)\otimes P_\Xi- (1+\widetilde{\cancel{\partial}}_0^2)^{-1}(\phi\otimes\psi^*)\otimes P_\Xi}
(f\otimes s\otimes v)\\
&\;\;\;\;= (1+\widetilde{\cancel{\partial}_R}^2)^{-1}\bra{(\phi\otimes\psi^*)\otimes P_\Xi-(1+\widetilde{\cancel{\partial}_R}^2)
(1+\widetilde{\cancel{\partial}}_0^2)^{-1}(\phi\otimes\psi^*)\otimes P_\Xi} (f\otimes s\otimes v)\\
&\;\;\;\;=(1+\widetilde{\cancel{\partial}_R}^2)^{-1}\innpro{\psi\otimes\Xi}{f}{}
\bra{
(\phi\otimes\Xi)\otimes s\otimes v-(1+\widetilde{\cancel{\partial}_R}^2)(1+\widetilde{\cancel{\partial}}^2_0)^{-1}
\Bigl(\bra{\phi\otimes\Xi}\otimes s\otimes v\Bigr)
} \\
&\;\;\;\;=: (1+\widetilde{\cancel{\partial}_R}^2)^{-1}{\rm Rem}.
\end{align*}

Since $(1+\widetilde{\cancel{\partial}_R}^2)^{-1}$ is a bounded operator, it is enough to study the remainder term ${\rm Rem}$. For this aim, we prepare a formula
\begin{align*}
\widetilde{\Dirac}_R^2&=D^2\otimes\id\otimes\id +2\sum n\Bigl(dR_{z_n}\otimes\id\otimes d\rho(\overline{z_n})+dR_{\overline{z_n}}\otimes\id\otimes\gamma(z_n)\Bigr) +\id\otimes_1\Dirac_L^2\\
&=:\partial_1+\partial_2+\partial_3,
\end{align*}
which can be obtained by a simple calculation. The formula simplifies the remainder term:
$${\rm Rem}=-\innpro{\psi\otimes\Xi}{f}{}\bra{(\partial_1+\partial_2)(1+\widetilde{\cancel{\partial}}^2_0)^{-1}\Bigl(\bra{\phi\otimes\Xi}\otimes s\otimes v\Bigr)
}.$$

Since the $\partial_1$-part is independent of the energy of $s\otimes v$, the composition $\partial_1\circ (1+\widetilde{\cancel{\partial}}^2_0)^{-1}$ is of the form
\begin{center}
rank one operator $\otimes$ compact operator.
\end{center}

For the $\partial_2$-part, recall Lemma \ref{estimate about PER}, that is, $d\rho$-part is much weaker than $\Dirac_L^2$. Thus, the composition
$$dR_{z_n}\otimes\id\otimes d\rho(\overline{z_n})\circ(1+\Dirac_0^2)^{-1}\circ a\otimes_1\id$$
is compact. Morevoer, thanks to $dR_{z_n}$ acting on $\Xi_{\sigma_n}$, the operator norm of the above is less than $\sqrt{n}\sigma_n\times \|d\rho(\overline{z_n})(1+\Dirac_0^2)^{-1}\|$ if $n$ is large enough. Hence the infinite sum of operators converges and compact thanks to the key condition (\ref{eq:behavior of sigma}).


\end{proof}

\subsection{The Mishchenko line bundle, the assembly map for $\Omega T_0$ and the analytic index}

Let us define a $K$-theory class which plays a role of Mishchenko line bundle, following the result of Lemma \ref{cut off=proj}.
\begin{dfn}
Let $[c]\in KK(\bb{C},\Omega T_0\ltimes C_0(\Omega T_0))$ be the $K$-theory class represented by the rank one projection $P_\Xi$ onto $\bb{C}\Xi$, where $\Xi=\Xi_{\sigma_1}\otimes \Xi_{\sigma_2}\otimes \cdots$. More precisely, $[c]$ is presented by $(P_\Xi*(\Omega T_0\ltimes C_0(\Omega T_0)),0)$ in the language of unbounded (in fact bounded) Kasparov modules.
\end{dfn}

We have reached the following construction of the assembly map for $\Omega T_0$.

\begin{dfn}
$\mu^{\Omega T_0}_\tau(x):=[c]\otimes_{\Omega T_0\ltimes C_0(\Omega T_0)}j^{\Omega T_0}_\tau(x)\in \Psi(\bb{C},\Omega T_0\ltimes_\tau \bb{C})$.
\end{dfn}

\begin{thm}

$\mu^{\Omega T_0}_\tau(x)=(S_{\Omega T_0}\otimes (\Omega T_0\ltimes_\tau\bb{C}),\Dirac_L)$ as an element of $\Psi(\bb{C},\Omega T_0\ltimes_\tau \bb{C})$

\end{thm}
\begin{proof}
The proof is almost parallel to Proposition \ref{value of the assembly map} except that the Dirac operator is an infinite sum. Just like Proposition \ref{convergence}, the infinite sum converges.
\end{proof}

This theorem looks quite natural. Let us recall the definition of the analytic index, Definition \ref{def of a-ind}. Let us define the analytic index $\ind_{\Omega T_0\ltimes_\tau \bb{C}}(x)\in \Psi(\bb{C},\Omega T_0\ltimes_\tau \bb{C})$ as follows.

\begin{lem}
$\Omega T_0\ltimes_{-\tau}\bb{C}=\ca{K}(L^2(\bb{R}^\infty))$ has an $\Omega T_0\ltimes_\tau\bb{C}=\ca{K}(L^2(\bb{R}^\infty)^*)$-Hilbert module structure defined by
$$f*b:={}^tbf$$
$$\innpro{f_1}{f_2}{}:={}^t(f_2f_1^*)$$
for $f$, $f_1$, $f_2\in \Omega T_0\ltimes_{-\tau}\bb{C}$ and $b\in \Omega T_0\ltimes_{\tau}\bb{C}$, just like Proposition \ref{small remark on a-ind}.
\end{lem}

\begin{dfn}
We define the analytic index of $x$ by 
$$\ind_{\Omega T_0\ltimes_\tau\bb{C}}(x):=((\Omega T_0\ltimes_{-\tau}\bb{C})\otimes S_{\Omega T_0},\Dirac_R)\in \Psi(\bb{C},\Omega T_0\ltimes_\tau\bb{C}).$$

The bimodule structure is defined by the above lemma.

\end{dfn}

Two indices $\mu^{\Omega T_0}_\tau(x)$ and $\ind_{\Omega T_0\ltimes_\tau\bb{C}}(x)$ coincide as follows.
\begin{thm}
$$\mu^{\Omega T_0}_\tau(x)=\ind_{\Omega T_0\ltimes_\tau\bb{C}}(x).$$
\end{thm}
\begin{proof}
The transpose $\Omega T_0\ltimes_{-\tau}\bb{C}= \ca{K}(L^2(\bb{R}^\infty))\to \ca{K}(L^2(\bb{R}^\infty)^*)=\Omega T_0\ltimes_{\tau}\bb{C}$ gives an isomorphism as $(\Omega T_0\ltimes_\tau \bb{C})$-Hilbert modules.
More concretely, $\phi\otimes \psi\in L^2(\bb{R}^\infty)\otimes^\alg L^2(\bb{R}^\infty)^*\subseteq \ca{K}(L^2(\bb{R}^\infty))$ is mapped to $\psi\otimes\phi\in L^2(\bb{R}^\infty)^*\otimes^\alg L^2(\bb{R}^\infty)\subseteq \ca{K}(L^2(\bb{R}^\infty)^*)$. Hence $\Dirac_L$ corresponds to $\Dirac_R$.

\end{proof}

\section*{Acknowledgements}
I am very grateful to my supervisors Nigel Higson and Tsuyoshi Kato. I also appreciate their current/former students who discussed with me. I benefit from their various back grounds.
I am supported by JSPS KAKENHI Grant Number 16J02214 and
the Kyoto Top Global University Project (KTGU).

\end{document}